\documentclass[twocolumn,final,amsthm]{autart}

\usepackage{amsmath}
\usepackage{amssymb}
\usepackage{amsthm}
\usepackage{graphicx}
\usepackage[dvipsnames]{xcolor}
\usepackage{pdfsync}

\usepackage[shortlabels]{enumitem}
    \setlist{nosep,leftmargin=*}

\theoremstyle{plain}
\newtheorem{Satz}{Theorem}
\newtheorem{Aussage}[Satz]{Proposition}
\newtheorem{Korollar}[Satz]{Corollary}
\newtheorem{Hilfssatz}{Lemma}

\theoremstyle{definition}
\newtheorem{Definition}{Definition}

\theoremstyle{remark}
\newtheorem{Bemerkung}{Remark}

\let\ForAll\forall
\renewcommand\forall{\ForAll\,}
\let\Exists\exists
\renewcommand\exists{\Exists\,}

\newcommand{\N}{\mathbb N}
\newcommand{\R}{\mathbb R}

\newcommand{\A}{\mathcal A}
\newcommand{\cC}{\mathcal C}
\newcommand{\cD}{\mathcal D}
\newcommand{\cH}{\mathcal H}
\newcommand{\cM}{\mathcal M}
\newcommand{\cS}{\mathcal S}
\newcommand{\cU}{\mathcal U}
\newcommand{\cX}{\mathcal X}
\newcommand{\cZ}{\mathcal Z}

\newcommand{\id}{\mathrm{id}}

\newcommand{\PD}{\mathcal{PD}}
\newcommand{\K}{\mathcal K }
\newcommand{\Kinf}{\mathcal K_\infty}
\newcommand{\Lfcn}{\mathcal L}

\newcommand{\KL}{\mathcal{KL}}
\newcommand{\KLL}{\mathcal{KLL}}

\newcommand{\Iff}{\Leftrightarrow}

\DeclareMathOperator*{\esssup}{\mathrm{ess\,sup}}
\DeclareMathOperator*{\argmax}{\mathrm{argmax}}
\let\limsup\relax
\DeclareMathOperator*{\limsup}{\overline{\lim}}
\DeclareMathOperator*{\interior}{\mathrm{int}}

\DeclareMathOperator{\dom}{\mathrm{dom}}
\DeclareMathOperator{\doms}{\mbox{\scriptsize \ensuremath{\mathrm{dom}}}}

\begin{document}
\begin{frontmatter}

\title{Lyapunov small-gain theorems for networks of\\
not necessarily ISS hybrid systems\thanksref{footnoteinfo}}

\thanks[footnoteinfo]{This paper was not presented at any IFAC meeting. Corresponding author A. Mironchenko. Tel. +49-851-509-3363.}

\author[Passau]{Andrii Mironchenko}\ead{andrii.mironchenko@uni-passau.de},
\author[UCSB]{Guosong Yang}\ead{guosongyang@ucsb.edu},
\author[UIUC]{Daniel Liberzon}\ead{liberzon@illinois.edu}

\address[Passau]{Faculty of Computer Science and Mathematics, University of Passau, Innstra\ss e 33, 94032 Passau, Germany}
\address[UCSB]{Department of Electrical and Computer Engineering, University of California, Santa Barbara, CA 93106 U.S.A.}
\address[UIUC]{Coordinated Science Laboratory, University of Illinois at Urbana-Champaign, 1308 W. Main St., Urbana, IL 61801 U.S.A.}


\begin{abstract}
We prove a novel Lyapunov-based small-gain theorem for networks composed of $ n \geq 2 $ hybrid subsystems which are not necessarily input-to-state stable. This result unifies and extends several small-gain theorems for hybrid and impulsive systems proposed in the last few years. We also show how average dwell-time (ADT) clocks and reverse ADT clocks can be used to modify the ISS Lyapunov functions for subsystems and to enlarge the applicability of the derived small-gain theorems.
\end{abstract}
\end{frontmatter}

\setlength{\abovedisplayskip}{6pt}  
\setlength{\belowdisplayskip}{6pt}  



\section{Introduction}\label{sec:intro}
The study of interconnections plays a significant role in the system theory, as it allows one to establish stability for a complex system based on properties of its less complex components. In this context, small-gain theorems prove to be useful and general in analyzing feedback interconnections, which are ubiquitous in the control literature. An overview of classical small-gain theorems involving input-output gains of linear systems can be found in \cite{DesoerVidyasagar2009}. In \cite{Hill1991,MareelsHill1992}, the small-gain technique was extended to nonlinear feedback systems within the input-output context. The next peak in the stability analysis of interconnections was reached based on the input-to-state stability (ISS) framework proposed in \cite{Sontag1989}, which unified the notions of internal and external stability. Nonlinear small-gain theorems for general feedback interconnections of two ISS systems were introduced in \cite{JiangTeelPraly1994,JiangMareelsWang1996}. Their generalization to networks composed of $ n \geq 2 $ ISS systems were reported in \cite{DashkovskiyRufferWirth2007,DashkovskiyRufferWirth2010}, with several variations summarized in \cite{DashkovskiyEfimovSontag2011}.

The results described above have been developed for continuous-time systems (i.e., ordinary differential equations). In the discrete-time context, small-gain theorems for general feedback interconnections of two ISS systems were established in \cite{JiangWang2001,LailaNesic2003}, and their generalization to networks composed of $ n \geq 2 $ ISS systems can be found in \cite{LiuJiangHill2012}. However, in modeling real-world phenomena one often has to consider interactions between continuous and discrete dynamics. A general framework for modeling such behaviors is the hybrid systems theory \cite{HaddadChellaboinaNersesov2006,GoebelSanfeliceTeel2012}. In this work, we adopt the hybrid system model in \cite{GoebelSanfeliceTeel2012}, which proves to be natural and general from the viewpoint of Lyapunov stability theory \cite{CaiTeelGoebel2007,CaiTeelGoebel2008}. The notions of input-to-state stability and ISS Lyapunov functions were extended for this class of hybrid systems in \cite{CaiTeel2009}.

Due to their interactive nature, many hybrid systems can be inherently modeled as feedback interconnections \cite[Section~V]{LiberzonNesicTeel2014}. During recent years, great efforts have been devoted to the development of small-gain theorems for interconnected hybrid systems. Trajectory-based small-gain theorems for interconnections of two hybrid systems were reported in \cite{NesicLiberzon2005,KarafyllisJiang2007,DashkovskiyKosmykov2013}, while Lyapunov-based formulations were proposed in \cite{LiberzonNesic2006,NesicTeel2008,LiberzonNesicTeel2014}. Some of these results were extended to networks composed of $ n \geq 2 $ ISS hybrid systems in \cite{DashkovskiyKosmykov2013}.

A more challenging problem is the study of hybrid systems in which either the continuous or the discrete dynamics is destabilizing (non-ISS). In this case, input-to-state stability is usually achieved under restrictions on the frequency of discrete events, such as dwell-time \cite{Morse1996}, average dwell-time (ADT) \cite{HespanhaMorse1999} and reverse average dwell-time (RADT) \cite{HespanhaLiberzonTeel2008}. For interconnections of such hybrid subsystems, the small-gain theorems established in \cite{DashkovskiyKosmykov2013,LiberzonNesicTeel2014} cannot be applied directly. The results of \cite{LiberzonNesicTeel2014} show that one can modify the non-ISS dynamics in subsystems by first adding auxiliary clocks and then constructing ISS Lyapunov functions for the augmented subsystems that decrease both during flow and at jumps. One advantage of this method is that it can be applied even if the non-ISS dynamics are of different types (i.e., if in some subsystems the continuous dynamics are non-ISS, and in some other ones the discrete dynamics are non-ISS). However, such modifications will lead to enlarged Lyapunov gains of subsystems, and hence make the small-gain condition more restrictive.

Another type of small-gain theorems was proposed in \cite{DashkovskiyKosmykovMironchenkoNaujok2012,DashkovskiyMironchenko2013SICON} for interconnected impulsive systems with continuous or discrete non-ISS dynamics. The first step in this method is to construct a candidate exponential ISS Lyapunov function for the interconnection. Provided that the non-ISS dynamics of subsystems are of the same type (i.e., when either the continuous dynamics of all subsystems or the discrete dynamics of all subsystems are ISS), the candidate exponential ISS Lyapunov function can be used to establish input-to-state stability of the interconnection under suitable ADT/RADT conditions. Compared with the previous method, this one doesn't require modifications of subsystems, and hence preserves the Lyapunov gains and validity of small-gain conditions. However, this method has been developed only for impulsive systems and requires candidate exponential ISS Lyapunov functions for subsystems. Moreover, it cannot be applied to interconnections of subsystems with different types of non-ISS dynamics.

In this paper, we unify the two methods above. In Section~\ref{sec:pre}, we introduce the modeling framework and main definitions, followed by a Lyapunov-based sufficient condition for ISS of hybrid systems with continuous or discrete non-ISS dynamics. In Section~\ref{sec:smg}, we establish a general small-gain theorem for an interconnection of $ n \geq 2 $ hybrid subsystems by constructing a candidate ISS Lyapunov function for the interconnection, which generalizes the Lyapunov-based small-gain theorems from \cite{NesicTeel2008,DashkovskiyKosmykovMironchenkoNaujok2012,DashkovskiyKosmykov2013,DashkovskiyMironchenko2013SICON,LiberzonNesicTeel2014}. We also derive several implications of the general result, in particular, a small-gain theorem for interconnections of subsystems with the same type of non-ISS dynamics and also candidate exponential ISS Lyapunov functions with linear Lyapunov gains. In Section~\ref{sec:aug}, we propose a version of the approach of modifying ISS Lyapunov functions for subsystems from \cite{LiberzonNesicTeel2014}, in which fewer subsystems are affected (and hence fewer Lyapunov gains are enlarged). In Section~\ref{sec:sum}, we summarize the results of this work as a unified method for establishing ISS of interconnections of hybrid subsystems and conclude the paper with an outlook on future research.

A preliminary and shortened version of the paper has been presented at the 21st International Symposium on Mathematical Theory of Networks and Systems \cite{MironchenkoYangLiberzon2014}.

\section{Framework for hybrid systems}\label{sec:pre}
Let $ \R_+ := [0, \infty) $ and $ \N := \{0, 1, 2, \ldots\} $. For a vector $ x \in \R^N $, denote by $ |x| $ its Euclidean norm, and by $ |x|_\A := \inf_{y \in \A} |x - y| $ its Euclidean distance to a set $ \A \subset \R^N $. For $ n $ vectors $ x_1, \ldots, x_n $, denote by $ (x_1, \ldots, x_n) := (x_1^\top, \ldots, x_n^\top)^\top $ their concatenation. For two vectors $ x, y \in \R^n $, we say that $ x \geq y $ and $ x > y $ if the corresponding inequality holds in all scalar components, and that $ x \ngeq y $ if there is at least one scalar component $ i $ in which $ x_i < y_i $. For a set $ \A $, denote by $ \overline\A $ and $ \interior\A $ its closure and interior, respectively.

Denote by $ \id $ the identity function. A function $ \alpha: \R_+ \to \R_+ $ is of class $ \PD $ if it is continuous and positive-definite (i.e., $ \alpha(r) = 0 \Iff r = 0 $); it is of class $ \K $ if $ \alpha \in \PD $ and is strictly increasing; it is of class $ \Kinf $ if $ \alpha \in K $ and is unbounded. A function $ \gamma: \R_+ \to \R_+ $ is of class $ \Lfcn $ if it is continuous, strictly decreasing and $ \lim_{t \to \infty} \gamma(t) = 0 $. A function $ \beta: \R_+ \times \R_+ \to \R_+ $ is of class $ \KL $ if $ \beta(\cdot, t) \in \K $ for each fixed $ t $ and $ \beta(r, \cdot) \in \Lfcn $ for each fixed $ r > 0 $.

Motivated by \cite{CaiTeel2009}, a hybrid system is modeled as the combination of a continuous flow and discrete jumps
\begin{equation}\label{eq:hyd}
\begin{aligned}
  &\dot x \in F(x,u), &\qquad &(x, u) \in \cC, \\
  &x^+ \in G(x,u), &\qquad &(x, u) \in \cD,
\end{aligned}
\end{equation}
where $ x \in \cX \subset \R^N $ is the state, $ u \in \cU \subset \R^M $ is the input, $ \cC \subset \cX \times \cU $ is the flow set, $ \cD \subset \cX \times \cU $ is the jump set, $ F: \cC \rightrightarrows \R^N $ is the flow map (here by $ \rightrightarrows $ we mean that $ F $ is a set-valued function, which maps each element of $ \cC $ to a subset of $ \R^N $), and $ G: \cD \rightrightarrows \cX $ is the jump map. (In this model, the dynamics of \eqref{eq:hyd} is continuous in $ \cC\backslash\cD $ and discrete in $ \cD\backslash\cC $. In $ \cC \cap \cD $, it can be either continuous or discrete.) The hybrid system \eqref{eq:hyd} is fully characterized by its \emph{data} $ \cH := (F, G, \cC, \cD, \cX, \cU) $.

Solutions of \eqref{eq:hyd} are defined on hybrid time domains. A set $ E \subset \R_+ \times \N $ is called a \emph{compact hybrid time domain} if $ E = \bigcup_{j=0}^{J} ([t_j, t_{j+1}], j) $ for some finite sequence of times $ 0 = t_0 \leq t_1 \leq \cdots \leq t_{J+1} $. It is a \emph{hybrid time domain} if $ E \cap ([0, T] \times \{0, 1, \ldots, J\}) $ is a compact hybrid time domain for each $ (T, J) \in E $. On a hybrid time domain, there is a natural ordering of points, that is, $ (s, k) \preceq (t, j) $ if $ s + k \leq t + j $, and $ (s, k) \prec (t, j) $ if $ s + k < t + j $.

Functions defined on hybrid time domains are called \emph{hybrid signals}. A hybrid signal $ x: \dom x \to \cX $ (defined on the hybrid time domain $ \dom x $) is a \emph{hybrid arc} if $ x(\cdot, j) $ is locally absolutely continuous for each $ j $. A hybrid signal $ u: \dom u \to \cU $ is a \emph{hybrid input} if $ u(\cdot, j) $ is Lebesgue measurable and locally essentially bounded for each $ j $. A hybrid arc $ x: \dom x \to \cX $ and a hybrid input $ u: \dom u \to \cU $ form a \emph{solution pair} $ (x, u) $ of \eqref{eq:hyd} if
\begin{itemize}
    \item $ \dom x = \dom u $ and $ (x(0, 0), u(0, 0)) \in \overline\cC \cup \cD $, where $ x(t, j) $ denotes the state of the hybrid system at hybrid time $ (t, j) $, that is, at time $ t $ and after $ j $ jumps;
    \item for each $ j \in \N $, it holds that $ (x(t, j), u(t, j)) \in \cC $ for all $ t \in \interior I_j $ and $ \dot x(t, j) \in F(x(t, j), u(t, j)) $ for almost all $ t \in I_j $, where $ I_j := \{t: (t, j) \in \dom x\} $;
    \item for each $ (t, j) \in \dom x $ such that $ (t, j+1) \in \dom x $, it holds that $ (x(t, j), u(t, j)) \in \cD $ and $ x(t, j+1) \in G(x(t, j), u(t, j)) $.
\end{itemize}
With proper assumptions on the data $ \cH $, one can establish local existence of solutions, which are not necessarily unique (see, e.g., \cite[Proposition~2.10]{GoebelSanfeliceTeel2012}). A solution pair $ (x, u) $ is \emph{maximal} if it cannot be extended, and \emph{complete} if $ \dom x $ is unbounded. In this paper, we only consider maximal solution pairs.

Following \cite{CaiTeel2009}, the essential supremum norm of a hybrid signal $ u $ up to a hybrid time $ (t, j) $ is defined by
\begin{equation*}
    \|u\|_{(t, j)} := \max \!\bigg\{ \!\esssup\limits_{\substack{(s, k) \in \doms u, \\ (s, k) \preceq (t, j)}}\! |u(s, k)|,\, \!\!\sup\limits_{\substack{(s, k) \in J(u), \\ (s, k) \preceq (t, j)}}\! |u(s, k)| \!\bigg\},
\end{equation*}
where $ J(x) := \{(s, k) \in \dom u : (s, k+1) \in \dom u\} $ is the set of jump times. In particular, the set of measure $ 0 $ of hybrid times that are ignored in computing the essential supremum norm cannot contain any jump time.

For a locally Lipschitz function $ V: \R^n \to \R $, its \emph{Dini derivative} at $ x \in \R^n $ in the direction $ y \in \R^n $ is given by
\begin{equation*}
  \dot V(x; y) := \limsup_{h \searrow 0} \dfrac{V(x + hy) - V(x)}{h},
\end{equation*}
where $ \limsup $ denotes the limit superior.

In this paper, we study input-to-state stability (ISS) properties of the hybrid system \eqref{eq:hyd} using ISS Lyapunov functions. Let $ \A \subset \cX $ be a compact set.
\begin{Definition}\label{dfn:hyd-iss}
Following \cite{LiberzonNesicTeel2014}, we say that a set of solution pairs $ \cS $ of \eqref{eq:hyd} is \emph{pre-input-to-state stable (pre-ISS) w.r.t. $ \A $} if there exist $ \beta \in \KL $ and $ \gamma \in \K $ such that for all $ (x, u) \in \cS $,
\begin{equation}\label{eq:hyd-iss-dfn}
    |x(t, j)|_\A \leq \max\{\beta(|x(0, 0)|_\A, t + j),\, \gamma(\|u\|_{(t, j)})\}
\end{equation}
for all $ (t, j) \in \dom x $. If $ \cS $ contains all solution pairs of \eqref{eq:hyd}, then we say that \eqref{eq:hyd} is \emph{pre-ISS w.r.t. $ \A $}. In addition, if all solution pairs are complete then we say that \eqref{eq:hyd} is \emph{ISS w.r.t. $ \A $}.
\end{Definition}
\begin{Bemerkung}\label{rmk:hyd-iss-gas}
If \eqref{eq:hyd-iss-dfn} holds with $ \gamma \equiv 0 $, then the set $ \cS $ is \emph{globally pre-asymptotically stable (pre-GAS)}, which implies that all complete solution pairs in $ \cS $ converge to $ \A $. In addition, if all solution pairs in $ \cS $ are complete then it is \emph{globally asymptotically stable (GAS)} \cite{LiberzonNesicTeel2014}.
\end{Bemerkung}
\begin{Bemerkung}\label{rmk:hyd-iss-kll}
In \cite{CaiTeel2009}, ISS of hybrid systems is defined in terms of class $ \KLL $ functions and without requiring all solution pairs to be complete, which is equivalent to our definition of pre-ISS with $ \KL $ functions \cite[Lemma~6.1]{CaiTeelGoebel2007}.
\end{Bemerkung}

\begin{Definition}\label{dfn:hyd-lya}
For the hybrid system \eqref{eq:hyd}, a function $ V: \cX \to \R_+ $ is a \emph{candidate ISS Lyapunov function w.r.t. $ \A $} if it is locally Lipschitz outside $ \A $,\footnote{The Lipschitz condition here is used to ensure the existence of the Dini derivative in \eqref{eq:hyd-lya-dfn-flow}, and it can be relaxed to that the function $ V $ is locally Lipschitz on an open set containing all $ x \notin \A $ such that $ (x, u) \in \cC $ for some $ u \in \cU $.\label{ftnt:lipschitz}} and
\begin{enumerate}[1.]
    \item there exist functions $ \psi_1, \psi_2 \in \Kinf $ such that
        \begin{equation}\label{eq:hyd-lya-dfn-bnd}
            \psi_1(|x|_\A) \leq V(x) \leq \psi_2(|x|_\A) \qquad \forall x \in \cX;
        \end{equation}
    \item there exist a gain function $ \chi \in \K $ and a continuous function $ \varphi: \R_+ \to \R $ with $ \varphi(0) = 0 $ such that for all $ (x, u) \in \cC $ with $ x \notin \A $,
        \begin{align}
            \hspace{-7mm} V(x) {\geq} \chi(|u|) \implies \dot V(x; y) {\leq} {-} \varphi(V(x)), \ y \in F(x, u); \label{eq:hyd-lya-dfn-flow}
        \end{align}
    \item there is a function $ \alpha \in \K $ such that for all $ (x, u) \in \cD $,\footnote{There is no loss of generality in requiring $ \alpha \in \K $ instead of $ \alpha \in \PD $, as a class $ \PD $ function can always be majorized by a class $ \K $ one. Meanwhile, $ \alpha \in \K $ is needed in establishing the small-gain theorems below, as explained in footnote~\ref{ftnt:jump-rate}.\label{ftnt:jump-rate-dfn}}
        \begin{align}
          \hspace{-7mm} V(x) \geq \chi(|u|) \implies V(y) \leq \alpha(V(x)),\ y \in G(x, u). \label{eq:hyd-lya-dfn-jump}
        \end{align}		
\end{enumerate}
In addition, if there exist two constants $ c, d \in \R $ so that
\begin{equation}\label{eq:hyd-lya-dfn-rate-coef}
    \varphi(r) \equiv c r, \quad \alpha(r) \equiv e^{-d} r
\end{equation}
in \eqref{eq:hyd-lya-dfn-flow} and \eqref{eq:hyd-lya-dfn-jump}, then $ V $ is a \emph{candidate exponential ISS Lyapunov function w.r.t. $ \A $} with \emph{rate coefficients} $ c, d $.
\end{Definition}

The next lemma gives an alternative characterization of the candidate ISS Lyapunov function, which will be useful in formulating the small-gain theorems in Section~\ref{sec:smg}.
\begin{Hilfssatz}\label{lem:hyd-lya-alt}
For the hybrid system \eqref{eq:hyd}, a function $ V: \cX \to \R_+ $ is a candidate ISS Lyapunov function w.r.t. $ \A $ if and only if it is locally Lipschitz outside $ \A $, and
\begin{enumerate}[1.]
    \item there exist functions $ \psi_1, \psi_2 \in \Kinf $ such that \eqref{eq:hyd-lya-dfn-bnd} holds;
    \item there exist a gain function $ \bar\chi \in \K $ and a continuous function $ \varphi: \R_+ \to \R $ with $ \varphi(0) = 0 $ such that for all $ (x, u) \in \cC $ with $ x \notin \A $,
		\begin{align}
            \hspace{-8mm} V(x) {\geq} \bar\chi(|u|) \implies \dot V(x; y) {\leq} {-}\varphi(V(x)),\  y \in F(x, u); \label{eq:hyd-lya-alt-flow}
        \end{align}
    \item there is a function $ \alpha \in \K $ such that for all $ (x, u) \in \cD $,
        \begin{equation}\label{eq:hyd-lya-alt-jump}
            V(y) \leq \max\{\alpha(V(x)),\, \bar\chi(|u|)\} \quad \forall y \in G(x, u).
        \end{equation}
\end{enumerate}
\end{Hilfssatz}
\begin{proof}
The proof is along the lines of the proof of \cite[Proposition~1]{DashkovskiyMironchenko2013SICON} for ISS Lyapunov functions for impulsive systems, and is omitted here.
\end{proof}
Exponential ISS Lyapunov functions can be characterized in a similar way. Note that the functions $ \chi $ in Definition~\ref{dfn:hyd-lya} and $ \bar\chi $ in Lemma~\ref{lem:hyd-lya-alt} are different in general.

The notion of candidate ISS Lyapunov function is defined to characterize the effect of destabilizing (non-ISS) dynamics in a hybrid system. In Definition~\ref{dfn:hyd-lya}, it is not required that $ \varphi \in \PD $ or $ \alpha < \id $ on $ (0, \infty) $, that is, $ V $ does not necessarily decrease along solutions of the hybrid system \eqref{eq:hyd}. If both of these conditions hold, then $ V $ becomes an \emph{ISS Lyapunov function}, and similar analysis to the proof of \cite[Proposition~2.7]{CaiTeel2009} can be used to show that \eqref{eq:hyd} is pre-ISS (note that ISS in \cite{CaiTeel2009} means pre-ISS in this paper; see Remark~\ref{rmk:hyd-iss-kll}). Moreover, if only one of them holds,\footnote{Namely, either the continuous or the discrete dynamics taken alone is ISS; see \cite{Sontag1989} and \cite{JiangWang2001} for the definitions of ISS for continuous and discrete dynamics, respectively.\label{ftnt:partial-iss}} we are still able to establish ISS for the sets of solution pairs satisfying suitable conditions on the density of jumps (i.e., the number of jumps per unit interval of continuous time).
\begin{Aussage}\label{prop:hyd-iss-partial}
Let $ V $ be a candidate exponential ISS Lyapunov function w.r.t. $ \A $ for the hybrid system \eqref{eq:hyd} with rate coefficients $ c, d $. For arbitrary constants $ \eta, \lambda, \mu > 0 $, denote by $ \cS[\eta,\lambda,\mu] $ the set of solution pairs $ (x, u) $ so that
\begin{equation}\label{eq:gen-adt}
    -(d - \eta)(j - k) - (c - \lambda)(t - s) \leq \mu
\end{equation}
for all $ (s, k) \preceq (t, j) $ in the hybrid time domain $ \dom x $. Then $ \cS[\eta,\lambda,\mu] $ is pre-ISS w.r.t. $ \A $.
\end{Aussage}
\begin{proof}
The proof is along the lines of the proof of \cite[Theorem~1]{HespanhaLiberzonTeel2008} for ISS of impulsive systems. Consider an arbitrary solution pair $ (x, u) \in \cS[\eta,\lambda,\mu] $. Let the function $ \chi $ be as in \eqref{eq:hyd-lya-dfn-flow} and \eqref{eq:hyd-lya-dfn-jump}. For all $ (t_0, j_0) \preceq (t_1, j_1) $ in $ \dom x $, if
\begin{equation}\label{eq:hyd-lya-gain}
    V(x(s, k)) \geq \chi(\|u\|_{(s, k)})
\end{equation}
for all $ (s, k) \in \dom x $ such that $ (t_0, j_0) \preceq (s, k) \preceq (t_1, j_1) $, then \eqref{eq:hyd-lya-dfn-flow}--\eqref{eq:hyd-lya-dfn-rate-coef} imply that
\begin{equation}\label{eq:hyd-lya-rate-sum}
\begin{split}
    V(x(t_1, j_1)) &\leq e^{-d (j_1 - j_0) - c (t_1 - t_0)} V(x(t_0, j_0)) \\
    &\leq e^{-\eta (j_1 - j_0) - \lambda (t_1 - t_0) + \mu} V(x(t_0, j_0)),
\end{split}
\end{equation}
where the last inequality follows from \eqref{eq:gen-adt}. Now consider an arbitrary $ (t, j) \in \dom x $. If \eqref{eq:hyd-lya-gain} holds for all $ (s, k) \preceq (t, j) $ in $ \dom x $, then \eqref{eq:hyd-lya-rate-sum}, together with \eqref{eq:hyd-lya-dfn-bnd}, implies that
\begin{equation}\label{eq:hyd-iss-partial-kl}
    |x(t, j)|_\A \leq \beta(|x(0, 0)|_\A, t + j)
\end{equation}
with the function $ \beta \in \KL $ defined by
\begin{equation}\label{eq:hyd-iss-partial-kl-fcn}
    \beta(r, l) := \psi_1^{-1} \big( e^{-l \min\{\eta,\, \lambda\} + \mu} \psi_2(r) \big).
\end{equation}
Otherwise, let
$$ (t', j') = \argmax_{\substack{(s, k) \in \doms x, \\ (s, k) \preceq (t, j)}} \{s + k : V(x(s, k)) \leq \chi(\|u\|_{(s, k)})\}. $$
Then \eqref{eq:hyd-lya-gain} holds for all $ (s, k) \in \dom x $ such that $ (t', j') \prec (s, k) \preceq (t, j) $; thus \eqref{eq:hyd-lya-rate-sum} implies that
\begin{equation*}
\begin{split}
    V(x(t, j)) &\leq e^{-\eta (j - j') - \lambda (t - t') + \mu} \max\{1, e^{-d}\} V(x(t', j')) \\
    &\leq e^\mu \max\{1,\, e^{-d}\} \chi(\|u\|_{(t', j')}) \\
    &\leq e^\mu \max\{1,\, e^{-d}\} \chi(\|u\|_{(t, j)}),
\end{split}
\end{equation*}
where the term $ \max\{1,\, e^{-d}\} $ is needed if $ (t', j'+1) \in \dom x $ with $ V(x(t', j')) < \chi(\|u\|_{(t', j')}) $ and $ V(x(t', j'+1)) > \chi(\|u\|_{(t', j'+1)}) $, and the second inequality is due to $ \eta, \lambda > 0 $. Hence from \eqref{eq:hyd-lya-dfn-bnd}, it follows that
\begin{equation}\label{eq:hyd-iss-partial-k}
    |x(t, j)|_\A \leq \gamma(\|u\|_{(t, j)})
\end{equation}
with the function $ \gamma \in \K $ defined by
$$ \gamma(r) := \psi_1^{-1} \big( e^\mu \max\{1,\, e^{-d}\} \chi(r) \big). $$
Combining \eqref{eq:hyd-iss-partial-kl} and \eqref{eq:hyd-iss-partial-k}, we obtain that \eqref{eq:hyd-iss-dfn} holds for all $ (x, u) \in \cS[\eta,\lambda,\mu] $ and all $ (t, j) \in \dom x $.
\end{proof}

\begin{Bemerkung}\label{rmk:hyd-iss-partial-marginal}
We observe that, if both $ c, d < 0 $, then the inequality \eqref{eq:gen-adt} cannot hold for any complete solution pair, since there is always a large enough $ t $ or $ j $ such that $ \eta j + \lambda t > \mu $. However, it may still hold for solution pairs defined on bounded hybrid time domains. Moreover, if $ c > 0 > d $, then the claim of Proposition~\ref{prop:hyd-iss-partial} also holds for $ \eta = 0 $. The proof remain unchanged except that the last inequality in \eqref{eq:hyd-lya-rate-sum} now becomes
\begin{equation*}
\begin{split}
    &\quad\,\, e^{-d(j_1 - j_0) - c(t_1 - t_0)} V(x(t_0, j_0)) \\
    &\leq e^{-\lambda(t_1 - t_0) + \mu} V(x(t_0, j_0)) \\
    &\leq e^{(\lambda^2/c - \lambda)(t_1-t_0) - \lambda^2 (t_1 - t_0)/c  + \mu} V(x(t_0, j_0)) \\
    &\leq e^{\lambda d(j_1 - j_0)/c - \lambda^2 (t_1 - t_0)/c + (1 + \lambda/c) \mu} V(x(t_0, j_0)),
\end{split}
\end{equation*}
where the first inequality follows from \eqref{eq:gen-adt} with $ \eta = 0 $, and the last one comes from the estimate
\begin{equation*}
    e^{(\lambda^2/c - \lambda)(t_1-t_0)} = e^{(\lambda/c)(\lambda - c)(t_1-t_0)} \leq e^{(\lambda/c)(d (j_1 - j_0) + \mu)},
\end{equation*}
and the definition \eqref{eq:hyd-iss-partial-kl-fcn} becomes
$$ \beta(r, l) := \psi_1^{-1} \big( e^{-l \min\{ -\lambda d/c,\, \lambda^2/c\} + (1 + \lambda/c)\, \mu} \psi_2(r) \big). $$
Analogously, if $ d > 0 > c $, then the claim of Proposition~\ref{prop:hyd-iss-partial} also holds for $ \lambda = 0 $.
\end{Bemerkung}

\begin{Bemerkung}\label{rmk:hyd-iss-partial-adt}
If $ c > 0 \geq d $, then we can divide both sides of \eqref{eq:gen-adt} by $ -(d - \eta) > 0 $ to transform it to an average dwell-time (ADT) condition \cite{HespanhaMorse1999}. Analogously, if $ d > 0 \geq c $, then we can divide both sides of \eqref{eq:gen-adt} by $ -(c - \lambda) > 0 $ to transform it to the reverse average dwell-time (RADT) condition \cite{HespanhaLiberzonTeel2008}.
\end{Bemerkung}

Given a candidate exponential ISS Lyapunov function with rate coefficients $ c > 0 $ and/or $ d > 0 $, we can determine pre-ISS sets of solution pairs via Proposition~\ref{prop:hyd-iss-partial}. In the following section, we investigate the formulation of such functions for interconnections of hybrid systems.

\section{Interconnections and small-gain theorems}\label{sec:smg}
We are interested in the case where the hybrid system \eqref{eq:hyd} is decomposed as
\begin{equation}\label{eq:inter}
\begin{aligned}
    &\dot x_i \in F_i(x, u),\quad i = 1, \ldots, n, &\qquad &(x, u) \in \cC, \\
    &x^+_i \in G_i(x, u),\quad i = 1, \ldots, n, &\qquad &(x, u) \in \cD,
\end{aligned}
\end{equation}
where $ x := (x_1, \ldots, x_n) \in \cX \subset \R^N $ with $ x_i \in \cX_i \subset \R^{N_i} $ is the state, $ u \in \cU \subset \R^M $ is the common (external) input, $ \cC := \cC_1 \times \cdots \times \cC_n \times \cC_u $ with $ \cC_i \subset \cX_i $ and $ \cC_u \subset \cU $ is the flow set, $ \cD := \cD_1 \times \cdots \times \cD_n \times \cD_u $ with $ \cD_i \subset \cX_i $ and $ \cD_u \subset \cU $ is the jump set, $ F := (F_1, \ldots, F_n) $ with $ F_i: \cC \rightrightarrows \R^{N_i} $ is the flow map, and $ G := (G_1, \ldots, G_n) $ with $ G_i: \cD \rightrightarrows \cX_i $ is the jump map. The dynamics of $ x_i $ is called the $ i $-th subsystem of \eqref{eq:inter} and is denoted by $ \Sigma_i $. The interconnection \eqref{eq:inter} is denoted by $ \Sigma $. For each $ \Sigma_i $, the states of other subsystems are treated as (internal) inputs.

Many systems with hybrid behaviors can be naturally transformed into the form of \eqref{eq:inter}. As demonstrated in \cite[Section~V]{LiberzonNesicTeel2014}, a networked control system can be treated as an interconnection of continuous states and hybrid errors due to the network protocol, and a quantized control system can be modeled as an interconnection of continuous states and a discrete quantizer. Moreover, the ``natural decomposition'' of a hybrid system \eqref{eq:hyd} as an interconnection of its continuous and discrete parts is often of interest as well.

\begin{Bemerkung}
In \eqref{eq:inter}, all the subsystems, as well as the interconnection, share the same flow set $ \cC $ and the same jump set $ \cD $, which justifies the view of \eqref{eq:inter} as an interconnection of $ n $ hybrid subsystems.
\end{Bemerkung}

\begin{Bemerkung}\label{rmk:inter-sub-lya}
Based on Lemma~\ref{lem:hyd-lya-alt} and standard considerations clarifying the influence of particular subsystems (see, e.g., \cite[Lemma~2.4.1]{Mironchenko2012}), one can show that a function $ V_i: \cX_i \to \R_+ $ is a candidate ISS Lyapunov function w.r.t. a set $ \A_i \subset \cX_i $ for the subsystem $ \Sigma_i $ iff $ V_i $ is locally Lipschitz outside $ \A_i $, and
\begin{enumerate}[1.]
    \item there exist $ \psi_{i1}, \psi_{i2} \in \Kinf $ such that
        \begin{equation}\label{eq:inter-sub-lya-bnd}
            \psi_{i1}(|x_i|_{\A_i}) \leq V_i(x_i) \leq \psi_{i2}(|x_i|_{\A_i}) \quad \forall x_i \in \cX_i;
        \end{equation}
    \item there exist \emph{internal gains} $ \chi_{ij} \in \K $ for $ j \neq i $ and $ \chi_{ii} \equiv 0 $, an \emph{external gain} $ \chi_i \in \K $, and a continuous function $ \varphi_i: \R_+ \to \R $ with $ \varphi_i(0) = 0 $ such that for all $ (x, u) \in \cC $ with $ x_i \notin \A_i $,
        \begin{equation}\label{eq:inter-sub-lya-gain}
            V_i(x_i) \geq \max \!\Big\{ \!\max_{j=1}^{n} \chi_{ij}(V_j(x_j)),\, \chi_i(|u|) \Big\}
        \end{equation}
        implies that
        \begin{equation}\label{eq:inter-sub-lya-flow}
            \dot V_i(x_i; y_i) \leq -\varphi_i(V_i(x_i)) \qquad \forall y_i \in F_i(x, u);
        \end{equation}
    \item there is a function $ \alpha_i \in \K $ such that for all $ (x, u) \in \cD $,
        \begin{multline}\label{eq:inter-sub-lya-jump}
            V_i(y_i) \leq \max \!\Big\{ \alpha_i(V_i(x_i)),\, \max_{j=1}^{n} \chi_{ij}(V_j(x_j)), \\
            \chi_i(|u|) \Big\} \qquad \forall y_i \in G_i(x, u).
        \end{multline}
\end{enumerate}
In addition, $ V_i $ is a candidate exponential ISS Lyapunov function w.r.t. $ \A_i $ with rate coefficients $ c_i, d_i $ iff
\begin{equation}\label{eq:inter-sub-lay-rate-coef}
    \varphi_i(r) \equiv c_i r, \quad \alpha_i(r) \equiv e^{-d_i} r.
\end{equation}
\end{Bemerkung}

Suppose that for each subsystem $ \Sigma_i $, a candidate ISS Lyapunov function $ V_i $ is given (for discussions regarding the existence of candidate exponential ISS Lyapunov functions for hybrid systems, see \cite[Theorem~8.1]{CaiTeelGoebel2007}, \cite[Section~2]{CaiTeel2009}, and \cite[Remark~3]{YangLiberzonMironchenko2016}). The question of whether the interconnection \eqref{eq:inter} is pre-ISS depends on properties of the \emph{gain operator} $ \Gamma: \R_+^n \to \R_+^n $ defined by
\begin{equation}\label{eq:gain-operator}
    \Gamma(r_1, \ldots, r_n) := \!\Big( \max_{j=1}^{n} \chi_{1j}(r_j), \ldots, \max_{j=1}^{n} \chi_{nj}(r_j) \Big).\!
\end{equation}
To construct a candidate ISS Lyapunov function for the interconnection \eqref{eq:inter}, we adopt the notion of $ \Omega $-path \cite{DashkovskiyRufferWirth2010}.
\begin{Definition}\label{dfn:omega-path}
Given a function $ \Gamma: \R_+^n \to \R_+^n $, a function $ \sigma := (\sigma_1, \ldots, \sigma_n) $ with $ \sigma_i \in \Kinf,\, i = 1, \ldots, n $ is called an \emph{$ \Omega $-path w.r.t. $ \Gamma $} if
\begin{enumerate}[1.]
    \item all $ \sigma_i^{-1} $ are locally Lipschitz on $ (0, \infty) $;
    \item for each compact set $ P \subset (0, \infty) $, there exist finite constants $ K_2 > K_1 > 0 $ such that for all $ i $,
        \begin{equation*}
            0 < K_1 \leq (\sigma_i^{-1})' \leq K_2
        \end{equation*}
        for all points of differentiability of $ \sigma_i^{-1} $ in $ P $;
    \item the function $ \Gamma $ is a contraction on $ \sigma(\cdot) $, that is,
        \begin{equation}\label{eq:omega-path-contraction}
            \Gamma(\sigma(r)) < \sigma(r) \qquad \forall r > 0.
        \end{equation}
\end{enumerate}
\end{Definition}

\begin{Bemerkung}\label{rmk:gain-operator-sum}
In this paper, we consider primarily $ \Omega $-paths w.r.t. the gain operator $ \Gamma $ defined by \eqref{eq:gain-operator}, due to the terms $ \max_{j=1}^n \chi_{ij}(V_j(x_j)) $ in \eqref{eq:inter-sub-lya-gain} and \eqref{eq:inter-sub-lya-jump} when formulating candidate ISS Lyapunov functions for subsystems (which will be clear from the statement and proof of Theorem~\ref{thm:inter-lya} below). However, there are other equivalent formulations of candidate ISS Lyapunov functions for subsystems, which will naturally lead to gain operators in different forms (see, e.g., \cite{DashkovskiyRufferWirth2007,DashkovskiyRufferWirth2010}). In particular, if \eqref{eq:inter-sub-lya-gain} and \eqref{eq:inter-sub-lya-jump} were formulated using $ \sum_{j=1}^n \chi_{ij}(V_j(x_j)) $ instead of $ \max_{j=1}^n \chi_{ij}(V_j(x_j)) $, we would arrive at the alternative gain operator $ \bar\Gamma: \R_+^n \to \R_+^n $ defined by
\begin{equation*}\label{operator_gamma_sum_type}
    \bar\Gamma(r_1, \ldots, r_n) := \bigg( \sum_{j=1}^{n} \chi_{1j}(r_j), \ldots, \sum_{j=1}^{n} \chi_{nj}(r_j) \!\bigg).
\end{equation*}
Compared with \eqref{eq:gain-operator}, we see that $ \Gamma(v) \leq \bar\Gamma(v) $ for all $ v \neq 0 $; thus every $ \Omega $-path w.r.t. $ \bar\Gamma $ is an $ \Omega $-path w.r.t. $ \Gamma $. This alternative construction will be useful in establishing Theorem~\ref{thm:inter-lya-lin} for the case of linear internal gains below.
\end{Bemerkung}

We say that a function $ \Gamma: \R_+^n \to \R_+^n $ satisfies the \emph{small-gain condition} if
\begin{equation}\label{eq:smg-dfn}
    \Gamma(v) \ngeq v \qquad \forall v \in \R_+^n \backslash \{0\},
\end{equation}
or equivalently,
$$ \Gamma(v) \geq v \iff v = 0. $$
As reported in \cite[Proposition~2.7 and Remark~2.8]{KarafyllisJiang2011} (see also \cite[Theorem~5.2]{DashkovskiyRufferWirth2010}), if \eqref{eq:smg-dfn} holds for the gain operator $ \Gamma $ defined by \eqref{eq:gain-operator}, then there exists an $ \Omega $-path $ \sigma $ w.r.t. $ \Gamma $. Furthermore, $ \sigma $ can be made {smooth on $ (0, \infty) $} via standard mollification arguments \cite[Appendix~B.2]{Grune2002Book}. In this case, we construct a candidate ISS Lyapunov function for the interconnection \eqref{eq:inter} based on those for the subsystems and the corresponding $ \Omega $-path.
\begin{Satz}\label{thm:inter-lya}
Consider the interconnection \eqref{eq:inter}. Suppose that each subsystem $ \Sigma_i $ admits a candidate ISS Lyapunov function $ V_i $ w.r.t. a set $ \A_i $ with the internal gains $ \chi_{ij} $ as in \eqref{eq:inter-sub-lya-gain}, and  the small-gain condition \eqref{eq:smg-dfn} holds for the gain operator $ \Gamma $ defined by \eqref{eq:gain-operator}. Let $ \sigma = (\sigma_1, \ldots, \sigma_n) $ be an $ \Omega $-path w.r.t. $ \Gamma $ which is smooth on $ (0, \infty) $. Then the function $ V: \cX \to \R_+ $ defined by
\begin{equation}\label{eq:inter-lya-fcn}
    V(x) := \max_{i=1}^{n} \sigma_i^{-1}(V_i(x_i))
\end{equation}
is a candidate ISS Lyapunov function w.r.t. the set $ \A := \A_1 \times \cdots \times \A_n $ for \eqref{eq:inter}.
\end{Satz}
\begin{proof}
As each $ \sigma_i \in \Kinf $ is smooth on $ (0, \infty) $ and each $ V_i $ is locally Lipschitz outside $ \A_i $, it follows that each $ \sigma_i^{-1} \circ V_i $ is locally Lipschitz outside $ \A_i $. Hence the function $ V $ defined by \eqref{eq:inter-lya-fcn} is locally Lipschitz outside $ \A $. In the following, we prove that it satisfies the conditions of Lemma~\ref{lem:hyd-lya-alt}, by combining and extending the arguments in the proofs of \cite[Theorem~5.3]{DashkovskiyRufferWirth2010} and \cite[Theorem~III.1]{LiberzonNesicTeel2014}.

First, consider the functions $ \psi_1, \psi_2$ defined by
\begin{equation*}
\begin{aligned}
    \psi_1(r) &:= \min_{i=1}^{n} \sigma_i^{-1}(\psi_{i1}(r/\sqrt{n})), & \ \ r\in\R_+,\\
    \psi_2(r) &:= \max_{i=1}^{n} \sigma_i^{-1}(\psi_{i2}(r)), & \ \ r\in\R_+
\end{aligned}
\end{equation*}
with $ \psi_{i1}, \psi_{i2} $ as in \eqref{eq:inter-sub-lya-bnd}. Since $ \sigma_i, \psi_{i1}, \psi_{i2} \in \Kinf $, we have that $ \psi_1, \psi_2 \in \Kinf $. Thus \eqref{eq:inter-sub-lya-bnd} implies \eqref{eq:hyd-lya-dfn-bnd}. In particular,
\begin{equation*}
\begin{split}
    &\psi_1(|x|_\A) \leq \min_{i=1}^{n} \sigma_i^{-1} \Big( \psi_{i1} \Big( \max_{j=1}^{n} |x_j|_{\A_j} \Big) \Big) \\
    &\leq \max_{j=1}^{n} \sigma_j^{-1}(\psi_{j1}(|x_j|_{\A_j})) \leq \max_{j=1}^{n} \sigma_j^{-1}(V_j(x_j)) = V(x).
\end{split}
\end{equation*}
Second, consider the gain function $ \bar\chi$ defined by
\begin{equation}\label{eq:inter-lya-gain-fcn}
    \bar\chi(r) := \max_{i=1}^{n} \sigma_i^{-1}(\chi_i(r)), \ \ r\in\R_+
\end{equation}
with $ \chi_i $ as in \eqref{eq:inter-sub-lya-gain}, and the function $ \varphi$ defined by
\begin{equation}\label{eq:inter-lya-flow-fcn}
    \varphi(r) := \min_{i=1}^{n}\, (\sigma_i^{-1})'(\sigma_i(r))\, \varphi_i(\sigma_i(r)),\ \ r\in\R_+
\end{equation}
with $ \varphi_i $ as in \eqref{eq:inter-sub-lya-flow}. As all $ \sigma_i \in \Kinf $ are smooth on $ (0, \infty) $, $ \chi_i \in \K $, and $ \varphi_i $ are continuous with $ \varphi_i(0) = 0 $, it follows that $ \bar\chi \in \K $ and $ \varphi $ is continuous with $ \varphi(0) = 0 $. Consider the sets $ \cM_i \subset \cX,\, i = 1, \ldots, n $ defined by
\begin{equation*}
    \cM_i := \Big\{ x \in \cX : \sigma_i^{-1}(V_i(x_i)) > \max_{j \neq i} \sigma_j^{-1}(V_j(x_j)) \Big\}.
\end{equation*}
The fact that all $ V_i $ and $ \sigma_i^{-1} $ are continuous implies that all $ \cM_i $ are open in $ \cX $, $ \cM_i \cap \cM_j = \emptyset $ for all $ j \neq i $, and $ \cX = \bigcup_{i=1}^{n} \overline{\cM_i} $, where $ \overline{\cM_i} $ is the closure of $ \cM_i $ in $ \cX $. Thus for each $ (x, u) \in \cC $ with $ x \notin \A $, there are two possibilities:
\begin{enumerate}[1),leftmargin=0\parindent,itemindent=*]
    \item There is a unique $ i \in \{1, \ldots, n\} $ s.t. $ x \in \cM_i $. Then
        \begin{equation}\label{eq:inter-lya-flow-case-1}
            V(x) = \sigma_i^{-1}(V_i(x_i)),
        \end{equation}
        and $ x_i \notin \A_i $ due to $ x \notin \A $. Hence
        \begin{align}
            &V_i(x_i) = \sigma_i(V(x)) \notag\\
            &\geq \max_{j=1}^{n} \chi_{ij}(\sigma_j(V(x))) \geq \max_{j=1}^{n} \chi_{ij}(V_j(x_j)), \label{eq:inter-lya-flow-gain-int}
        \end{align}
        where the first inequality follows from \eqref{eq:omega-path-contraction}, and the second one follows from \eqref{eq:inter-lya-fcn}. Also, if $ V(x) \geq \bar\chi(|u|) $, then $ V(x) \geq \max_{j=1}^{n} \sigma_j^{-1} (\chi_j(|u|)) $ due to \eqref{eq:inter-lya-gain-fcn}; thus
        \begin{align}
            &V_i(x_i) = \sigma_i(V(x)) \geq \sigma_i \Big( \max_{j=1}^{n} \sigma_j^{-1} (\chi_j(|u|)) \Big) \notag\\
            &\geq \sigma_i (\sigma_i^{-1} (\chi_i(|u|))) = \chi_i(|u|). \label{eq:inter-lya-flow-gain-ext}
        \end{align}
        Hence \eqref{eq:inter-sub-lya-gain}, and therefore \eqref{eq:inter-sub-lya-flow}, holds. Given an arbitrary $ y = (y_1, \ldots, y_n) \in F(x, u) $, as $ \cM_i $ is open, it follows that $ x + h y \in \cM_i $ for all small enough $ h > 0 $; thus $ V(x + h y) = \sigma_i^{-1}(V_i(x_i + h y_i)) $. Hence
        \begin{equation*}
        \begin{split}
            \dot V(x; y)
            &= \limsup_{h \searrow 0} \dfrac{V(x + h y) - V(x)}{h} \\
            &= \limsup_{h \searrow 0} \dfrac{\sigma_i^{-1}(V_i(x_i + h y_i)) - \sigma_i^{-1}(V_i(x_i))}{h} \\
            &= (\sigma_i^{-1})'(V_i(x_i)) \limsup_{h \searrow 0} \dfrac{V_i(x_i + h y_i) - V_i(x_i)}{h} \\
            &= (\sigma_i^{-1})'(V_i(x_i)) \dot V_i (x_i; y_i) \\
            &\leq -(\sigma_i^{-1})'(\sigma_i(V(x)))\, \varphi_i(\sigma_i(V(x))) \\
            &\leq -\varphi(V(x)),
        \end{split}
        \end{equation*}
        where the first inequality follows from \eqref{eq:inter-sub-lya-flow} and \eqref{eq:inter-lya-flow-case-1}, and the last one follows from \eqref{eq:inter-lya-flow-fcn}.
    \item There is a subset $ I(x) \subset \{1, \ldots, n\} $ of indices with the cardinality $ |I(x)| \geq 2 $ such that $ x \in \bigcap_{i \in I(x)} \partial \cM_i $, where $ \partial\cM_i $ denotes the boundary of $ \cM_i $ in $ \cX $ and satisfies that $ \partial\cM_i = \overline{\cM_i} \backslash \cM_i $ as $ \cM_i $ is open in $ \cX $. Then \eqref{eq:inter-lya-flow-case-1} and $ x_i \notin \A_i $ hold for all $ i \in I(x) $. Following similar arguments to those in the previous case, if $ V(x) \geq \bar\chi(|u|) $, then \eqref{eq:inter-lya-flow-gain-int} and \eqref{eq:inter-lya-flow-gain-ext}, and therefore \eqref{eq:inter-sub-lya-flow}, hold for all $ i \in I(x) $. Given an arbitrary $ y = (y_1, \ldots, y_n) \in F(x, u) $, as all $ \cM_i $ are open, it follows that $ x + h y \in \big( \bigcap_{i \in I(x)} \partial\cM_i \big) \cap \big( \bigcap_{i \in I(x)} \cM_i \big) $ for all small enough $ h > 0 $; thus $ V(x + h y) = \max_{i \in I(x)} \sigma_i^{-1}(V_i(x_i + h y_i)) $. Hence
        \begin{equation*}
        \begin{split}
            \dot V(x; y)
            &= \limsup_{h \searrow 0} \dfrac{V(x + h y) - V(x)}{h} \\
            &= \limsup_{h \searrow 0} \dfrac{1}{h} \Big( \max_{i \in I(x)} \sigma_i^{-1}(V_i(x_i + h y_i)) - V(x) \Big) \\
            &= \limsup_{h \searrow 0} \max_{i \in I(x)} \dfrac{\sigma_i^{-1}(V_i(x_i + h y_i)) - \sigma_i^{-1}(V_i(x_i))}{h} \\
            &= \max_{i \in I(x)} \limsup_{h \searrow 0} \dfrac{\sigma_i^{-1}(V_i(x_i + h y_i)) - \sigma_i^{-1}(V_i(x_i))}{h} \\
            &= \max_{i \in I(x)} (\sigma_i^{-1})'(V_i(x_i)) \dot V_i(x_i; y_i) \\
            &\leq \max_{i \in I(x)} -(\sigma_i^{-1})'(\sigma_i(V(x)))\, \varphi_i(\sigma_i(V(x))) \\
            &\leq -\varphi(V(x)),
        \end{split}
        \end{equation*}
        where the fourth equality follows partially from the continuity of all $ V_i $ and $ \sigma_i^{-1} $ (cf. the proof of \cite[Theorem~4]{DashkovskiyMironchenko2013}); the first inequality follows from \eqref{eq:inter-sub-lya-flow} and \eqref{eq:inter-lya-flow-case-1} for $ i \in I(x) $, and the last one follows from \eqref{eq:inter-lya-flow-fcn}.
\end{enumerate}
Hence \eqref{eq:hyd-lya-alt-flow} holds for each $ (x, u) \in \cC $.

Last, consider the function $ \alpha: \R_+ \to \R_+ $ defined by
\begin{equation}\label{eq:inter-lya-jump-fcn}
    \alpha(r) := \max_{i, j=1}^{n} \!\Big\{ \sigma_i^{-1}(\alpha_i(\sigma_i(r))),\, \sigma_i^{-1}(\chi_{ij}(\sigma_j(r))) \Big\}
\end{equation}
with $ \alpha_i $ and $ \chi_{ij} $ as in \eqref{eq:inter-sub-lya-jump}. As all $ \sigma_i \in \Kinf $, $ \chi_{ij} \in \K $ for $ j \neq i $, $ \chi_{ii} \equiv 0 $, and $ \alpha_i \in \K $, it follows that $ \alpha \in \K $. Consider an arbitrary $ (x, u) \in \cD $. From \eqref{eq:inter-lya-fcn} and \eqref{eq:inter-lya-jump-fcn}, it follows that\footnote{Note that, if $ \alpha_i $ is of class $ \PD $ but not increasing, then it is possible that $ \sigma_i(V(x)) > V_i(x_i) $ but $ \alpha_i(\sigma_i(V(x))) < \alpha_i(V_i(x_i)) $ for some $ i $; thus the inequality following this footnote may not hold. A similar issue arises in the proof of \cite[Theroem~III.1]{LiberzonNesicTeel2014} where it was overlooked, but could be fixed by majorizing the class $ \PD $ functions $ \lambda_1, \lambda_2 $ with class $ \K $ ones.\label{ftnt:jump-rate}}
\begin{equation*}
    \alpha(V(x)) \geq \max_{i, j=1}^{n} \!\Big\{ \sigma_i^{-1}(\alpha_i(V_i(x_i))),\, \sigma_i^{-1}(\chi_{ij}(V_j(x_j))) \Big\}.
\end{equation*}
Also, \eqref{eq:inter-lya-gain-fcn} implies that $ \bar\chi(|u|) = \max_{i=1}^n \sigma_i^{-1}(\chi_i(|u|)) $. Combining the previous two equations with \eqref{eq:inter-sub-lya-jump}, we obtain that for all $ y = (y_1, \ldots, y_n) \in G(x, u) $,
$$ V(y) = \max_{i=1}^{n} \sigma_i^{-1}(V_i(y_i)) \leq \max\{\alpha(V(x)),\, \bar\chi(|u|)\}. $$
Hence \eqref{eq:hyd-lya-alt-jump} holds for each $ (x, u) \in \cD $.

Therefore, from Lemma~\ref{lem:hyd-lya-alt}, it follows that $ V $ is a candidate ISS Lyapunov function w.r.t. $ \A $ for \eqref{eq:inter}.
\end{proof}

Theorem~\ref{thm:inter-lya} is a powerful tool in establishing ISS of interconnections of hybrid subsystems. In the following, we inspect some of its implications.

If each subsystem of \eqref{eq:inter} admits an ISS Lyapunov function, then Theorem~\ref{thm:inter-lya} implies the following result, which generalizes \cite[Theorem~III.1]{LiberzonNesicTeel2014} and \cite[Theorem~3.6]{DashkovskiyKosmykov2013}.
\begin{Korollar}\label{cor:inter-iss}
Consider the interconnection \eqref{eq:inter}. Suppose that each subsystem $ \Sigma_i $ admits an ISS Lyapunov function $ V_i $ w.r.t. a set $ \A_i $ (i.e., $ \varphi_i \in \PD $ and $ \alpha_i < \id $ on $ (0, \infty) $ in \eqref{eq:inter-sub-lya-flow} and \eqref{eq:inter-sub-lya-jump}, respectively) with the internal gains $ \chi_{ij} $ as in \eqref{eq:inter-sub-lya-gain}, and the small-gain condition \eqref{eq:smg-dfn} holds for the gain operator $ \Gamma $ defined by \eqref{eq:gain-operator}. Then \eqref{eq:inter} is pre-ISS w.r.t. $ \A $.
\end{Korollar}
\begin{proof}
Following Theorem~\ref{thm:inter-lya}, the function $ V $ defined by \eqref{eq:inter-lya-fcn} is a candidate ISS Lyapunov function w.r.t. $ \A $ for \eqref{eq:inter}. First, as all $ \sigma_i \in \Kinf $ are smooth on $ (0, \infty) $ and $ \varphi_i \in \PD $, the function $ \varphi $ defined by \eqref{eq:inter-lya-flow-fcn} is of class $ \PD $. Second, \eqref{eq:omega-path-contraction} implies that all $ \sigma_i^{-1} \circ \chi_{ij} \circ \sigma_j < \id $ on $ (0, \infty) $, and as all $ \sigma_i \in \Kinf $ and $ \alpha_i < \id $ on $ (0, \infty) $, it follows that all $ \sigma_i^{-1} \circ \alpha_i \circ \sigma_i < \id $ on $ (0, \infty) $; thus the function $ \alpha $ defined by \eqref{eq:inter-lya-jump-fcn} satisfies that $ \alpha < \id $ on $ (0, \infty) $. Therefore, $ V $ is an ISS Lyapunov function, and \eqref{eq:inter} is pre-ISS w.r.t. $ \A $ following similar analysis to the proof of \cite[Proposition~2.7]{CaiTeel2009}; see also Remark~\ref{rmk:hyd-iss-kll}.
\end{proof}

As the assumptions in Corollary~\ref{cor:inter-iss} are quite restrictive, we now investigate the case where, for some subsystems $ \Sigma_i $, either $ \varphi_i \notin \PD $ or $ \alpha_i(r) \geq r $ for some $ r > 0 $ (cf. footnote~\ref{ftnt:partial-iss}). In this case, we cannot use Corollary~\ref{cor:inter-iss} to prove pre-ISS for the interconnection \eqref{eq:inter} directly, but rather invoke Proposition~\ref{prop:hyd-iss-partial} to establish pre-ISS for the set of solution pairs that jump neither too fast nor too slowly. However, in general, Theorem~\ref{thm:inter-lya} cannot provide the candidate exponential ISS Lyapunov function needed in Proposition~\ref{prop:hyd-iss-partial}. In the next theorem, we construct such a function under the assumption that each subsystem $ \Sigma_i $ admits a candidate exponential ISS Lyapunov function $ V_i $, and the internal gains $ \chi_{ij} $ in \eqref{eq:inter-sub-lya-gain} and \eqref{eq:inter-sub-lya-jump} are all linear. With a slight abuse of notation, we let all $ \chi_{ij} \geq 0 $ be scalars, and replace the terms $ \chi_{ij}(V_j(x_j)) $ in \eqref{eq:inter-sub-lya-gain} and \eqref{eq:inter-sub-lya-jump} with $ \chi_{ij} V_j(x_j) $. Consider the \mbox{\emph{gain matrix}} 
\begin{equation}\label{eq:gain-mx}
    \Gamma_M := (\chi_{ij})_{i,j=1}^{n} \in \R^{n \times n}.
\end{equation}
Denote by $ \rho(\Gamma_M) $ its spectral radius (i.e., the largest absolute value of its eigenvalues). Due to \cite[p.~110]{DashkovskiyRufferWirth2007}, if
\begin{equation}\label{eq:smg-mx}
    \rho(\Gamma_M) < 1,
\end{equation}
then the small-gain condition \eqref{eq:smg-dfn} holds for the function $ \bar\Gamma: \R^n_+ \to \R^n_+ $ defined by $ \bar\Gamma(v) := \Gamma_M v $, which is the alternative gain operator in Remark~\ref{rmk:gain-operator-sum}. Consequently, there exists a linear $ \Omega $-path w.r.t. $ \bar\Gamma $ \cite[p.~78]{DashkovskiyRufferWirth2006}; for more results on $ \Omega $-paths, the reader may consult \cite{Ruffer2010}.
\begin{Satz}\label{thm:inter-lya-lin}
Consider the interconnection \eqref{eq:inter}. Suppose that each subsystem $ \Sigma_i $ admits a candidate exponential ISS Lyapunov function $ V_i $ w.r.t. a set $ \A_i $ with rate coefficients $ c_i, d_i $. Assume also that the internal gains $ \chi_{ij} $ in \eqref{eq:inter-sub-lya-gain} and \eqref{eq:inter-sub-lya-jump} are all linear, and \eqref{eq:smg-mx} holds for the gain matrix $ \Gamma_M $ defined by \eqref{eq:gain-mx}. Let $ \sigma: r \mapsto (s_1 r, \ldots, s_n r) $ with scalars $ s_1, \ldots, s_n $ be a linear $ \Omega $-path w.r.t. the alternative gain operator $ \bar\Gamma $. Then $ V : \cX \to \R_+ $ defined by
\begin{equation}\label{eq:inter-lya-lin-fcn}
    V(x) := \max_{i=1}^{n} \dfrac{1}{s_i} V_i(x_i)
\end{equation}
is a candidate exponential ISS Lyapunov function w.r.t. $ \A $ for \eqref{eq:inter} with rate coefficients
\begin{equation}\label{eq:inter-lya-lin-rate-coef}
    c := \min_{i=1}^{n} c_i, \quad d := \min_{i, j: j \neq i} \!\bigg\{ d_i,\, -\ln \bigg( \dfrac{s_j}{s_i} \chi_{ij} \bigg) \!\bigg\}.
\end{equation}
\end{Satz}
\begin{proof}
In view of Remark~\ref{rmk:gain-operator-sum}, $ \sigma $ is also an $ \Omega $-path w.r.t. the gain operator defined by \eqref{eq:gain-operator} (with all $ \chi_{ij}(r_j) $ replaced by $ \chi_{ij} r_j $). Following Theorem~\ref{thm:inter-lya}, the function $ V $ defined by \eqref{eq:inter-lya-lin-fcn} is a candidate ISS Lyapunov function w.r.t. $ \A $ for \eqref{eq:inter}. Substituting \eqref{eq:inter-sub-lay-rate-coef} into \eqref{eq:inter-lya-flow-fcn} and \eqref{eq:inter-lya-jump-fcn}, we obtain
\begin{equation*}
    \varphi(r) \equiv \min_{i=1}^{n} c_i r, \quad \alpha(r) \equiv \max_{i, j = 1}^{n} \!\bigg\{ e^{-d_i},\, \dfrac{s_j}{s_i}\chi_{ij} \bigg\} r.
\end{equation*}
Hence $ V $ is a candidate exponential ISS Lyapunov function with the rate coefficients $ c, d $ defined by \eqref{eq:inter-lya-lin-rate-coef}.
\end{proof}
\begin{Bemerkung}\label{rmk:power-fcn}
For the more general case with the internal gains $ \chi_{ij} $ being power functions instead of linear ones, a candidate exponential ISS Lyapunov function for \eqref{eq:inter} can be constructed in a similar way; cf. \cite[Theorem~9]{DashkovskiyMironchenko2013SICON}.
\end{Bemerkung}

The following remark provides a simpler bound for the rate coefficient $ d $ in some important cases.
\begin{Bemerkung}
If the gain matrix $ \Gamma_M $ defined by \eqref{eq:gain-mx} is irreducible, then $ \rho(\Gamma_M) $ is the Perron--Frobenius eigenvalue of $ \Gamma_M $, and the corresponding eigenvector $ \bar s = (s_1, \ldots, s_n) $ satisfies $ \bar s > 0 $ (Perron--Frobenius theorem \cite[Theorem~2.1.3]{BermanPlemmons1994}). Hence, if \eqref{eq:smg-mx} holds, then $ \Gamma_M \bar s = \rho(\Gamma_M) \bar s < \bar s$; thus $ \sigma: r \mapsto \bar s r $ is a linear $ \Omega $-path as in Theorem~\ref{thm:inter-lya-lin}. Moreover, for all $ i \in \{1, \ldots, n\} $, it holds that
\begin{center}
$\max_{j=1}^{n} \dfrac{s_j}{s_i} \chi_{ij} \leq \dfrac{1}{s_i} \sum_{j=1}^{n} s_j \chi_{ij} = \rho(\Gamma_M);$
\end{center}
thus the rate coefficient $ d $ defined by \eqref{eq:inter-lya-lin-rate-coef} satisfies that $ d \geq \min \{\min_{i=1}^{n} d_i,\, -\ln(\rho(\Gamma_M))\} $.
\end{Bemerkung}

Having applied Theorem~\ref{thm:inter-lya-lin}, we can establish pre-ISS for the set of solution pairs that jump neither too fast nor too slowly via Proposition~\ref{prop:hyd-iss-partial}. However, if there are subsystems $ \Sigma_k, \Sigma_l $ for which the rate coefficients $ c_k, d_l < 0 $, then $ c, d $ defined by \eqref{eq:inter-lya-lin-rate-coef} are negative as well, and Proposition~\ref{prop:hyd-iss-partial} cannot be applied to complete solution pairs (see Remark~\ref{rmk:hyd-iss-partial-marginal}). In the following section, we handle such cases via the approach of modifying ISS Lyapunov functions for subsystems using ADT and RADT clocks from \cite{LiberzonNesicTeel2014}.

\section{Modifying ISS Lyapunov functions for subsystems}\label{sec:aug}
Suppose that each subsystem $ \Sigma_i $ admits a candidate exponential ISS Lyapunov function with rate coefficients $ c_i, d_i $, and there are $ \Sigma_k, \Sigma_l $ such that $ c_k, d_l < 0 < c_l, d_k $. Our goal is to construct new candidate exponential ISS Lyapunov functions with rate coefficients $ \tilde c_i, \tilde d_i $ so that either all $ \tilde c_i > 0 $ (i.e., all continuous dynamics are ISS) or all $ \tilde d_i > 0 $ (i.e., all discrete dynamics are ISS). To accomplish this, we first derive suitable conditions on the density of jumps, then augment the corresponding subsystems with auxiliary clocks to incorporate such conditions, and finally modify the corresponding candidate exponential ISS Lyapunov functions.

\subsection{Making discrete dynamics ISS}\label{ssec:aug-adt}
In the following, we construct candidate exponential ISS Lyapunov functions so that all rate coefficients $ \tilde d_i > 0 $.

We say that a solution pair $ (x, u) $ of \eqref{eq:inter} admits an average dwell-time (ADT) \cite{HespanhaMorse1999} $ \delta > 0 $ if there is an integer $ N_0 \geq 1 $ so that all $ (s, k) \preceq (t, j) $ in $ \dom x $ satisfy\footnote{If \eqref{eq:adt} holds with $ N_0 = 1 $, then the ADT condition becomes the dwell-time condition \cite{Morse1996}; if it holds with $ N_0 < 1 $, then jumps are not allowed at all, which can be seen directly from \eqref{eq:adt} by taking $ t - s $ small enough.}
\begin{equation}\label{eq:adt}
    j - k \leq \delta (t - s) + N_0.
\end{equation}
Following \cite[Section~IV.A]{LiberzonNesicTeel2014}, a hybrid time domain satisfies \eqref{eq:adt} iff it is the domain of an ADT clock $ \tau $ given by
\begin{equation}\label{eq:adt-clock}
\begin{aligned}
    &\dot\tau \in [0, \delta], &\qquad &\tau \in [0, N_0], \\
    &\tau^+ = \tau - 1, &\qquad &\tau_i \in [1, N_0].
\end{aligned}
\end{equation}

\begin{Bemerkung}\label{rmk:adt-clock-dfn}
This notion of ADT clock for hybrid systems first appeared in \cite[Appendix]{CaiTeelGoebel2008}, where it was defined by
\begin{equation}\label{eq:adt-clock-alt}
\begin{cases}
    \dot\tau \in \eta_\delta(\tau) &\text{for } \tau \in C := [0, N_0] \\
    \tau^+ = \tau - 1 &\text{for } \tau \in D := [1, N_0]
\end{cases}
\end{equation}
with \quad
$\eta_\delta(\tau) :=
    \begin{cases}
        \delta &\text{for } \tau \in [0, N_0) \\
        [0, \delta] &\text{for } \tau = N_0
    \end{cases}$\\
(see also \cite{MitraLiberzonLynch2008} for a related earlier construction). The ADT clocks defined by \eqref{eq:adt-clock} and \eqref{eq:adt-clock-alt} are equivalent in the following sense. First, as $ \tau \in [0, \delta] $, an ADT clock defined by \eqref{eq:adt-clock-alt} always satisfies \eqref{eq:adt-clock}. Second, given an ADT clock defined by \eqref{eq:adt-clock} that increases on $ [0, N_0) $ with a speed $ \dot\tau < \delta $, there always exists an ADT clock defined by \eqref{eq:adt-clock-alt} that increases on $ [0, N_0) $ with $ \dot\tau = \delta $ but stays longer at $ N_0 $ so that their hybrid time domains are the same.
\end{Bemerkung}

Denote by $ I_d := \{i : d_i < 0\} $ the index set of subsystems with non-ISS discrete dynamics. Let $ z_i := x_i \in \cX_i =: \cZ_i $ for $ i \notin I_d $ and $ z_i := (x_i, \tau_i) \in \cX_i \times [0, N_{0i}] =: \cZ_i $ with an integer $ N_{0i} \geq 1 $ for $ i \in I_d $. Consider the augmented interconnection $ \tilde\Sigma $ with state $ z := (z_1, \ldots, z_n) \in \cZ_1 \times \cdots \times \cZ_n =: \cZ $ and input $ u \in \cU $ modeled by
\begin{equation}\label{eq:aug-inter}
\begin{aligned}
    &\dot z_i \in \tilde F_i(z, u),\quad i = 1, \ldots, n, &\qquad &(z, u) \in \tilde\cC, \\
    &z_i^+ \in \tilde G_i(z, u),\quad i = 1, \ldots, n, &\qquad &(z, u) \in \tilde\cD,
\end{aligned}
\end{equation}
where $ \tilde\cC := \tilde\cC_1 \times \cdots \times \tilde\cC_n \times \cC_u $ with $ \tilde\cC_i := \cC_i $ for $ i \notin I_d $ and $ \tilde\cC_i := \cC_i \times [0, N_{0i}] $ for $ i \in I_d $, $ \tilde\cD := \tilde\cD_1 \times \cdots \times \tilde\cD_n \times \cD_u $ with $ \tilde\cD_i := \cD_i $ for $ i \notin I_d $ and $ \tilde\cD_i := \cD_i \times [1, N_{0i}] $ for $ i \in I_d $, $ \tilde F := (\tilde F_1, \ldots, \tilde F_n) $ with $ \tilde F_i(z, u) := F_i(x, u) $ for $ i \notin I_d $ and $ \tilde F_i(z, u) := F_i(x, u) \times [0, \delta_i] $ for $ i \in I_d $, and $ \tilde G := (\tilde G_1, \ldots, \tilde G_n) $ with $ \tilde G_i(z, u) := G_i(x, u) $ for $ i \notin I_d $ and $ \tilde G_i(z, u) := G_i(x, u) \times \{\tau_i - 1\} $ for $ i \in I_d $. Then \eqref{eq:aug-inter} is a hybrid system with the data $ \tilde\cH := (\tilde F, \tilde G, \tilde\cC, \tilde\cD, \cZ, \cU) $. The dynamics of $ z_i $ is called the $ i $-th augmented subsystem of \eqref{eq:aug-inter} and is denoted by $ \tilde\Sigma_i $.

In the following proposition, we apply the modification technique from \cite[Proposition~IV.1]{LiberzonNesicTeel2014} to construct a candidate exponential ISS Lyapunov function for each augmented subsystem $ \tilde\Sigma_i $ based on the candidate exponential ISS Lyapunov function for the subsystem $ \Sigma_i $ of the original interconnection \eqref{eq:inter} and the ADT clock $ \tau_i $.
\begin{Aussage}\label{prop:aug-inter-adt-sub-lya}
Consider a subsystem $ \Sigma_i $ of the original interconnection \eqref{eq:inter}. Suppose that it admits a candidate exponential ISS Lyapunov function $ V_i $ w.r.t. a set $ \A_i $ with rate coefficients $ c_i, d_i $. For a scalar $ L_i \geq 0 $, the function $ W_i: \cZ_i \to \R_+ $ defined by
\begin{equation*}
    W_i(z_i) :=
    \begin{cases}
        V_i(x_i) &\text{if } i \notin I_d; \\
        e^{L_i \tau_i} V_i(x_i) &\text{if } i \in I_d
    \end{cases}
\end{equation*}
is a candidate exponential ISS Lyapunov function w.r.t. 
\begin{equation*}
    \tilde\A_i :=
    \begin{cases}
        \A_i &\text{if } i \notin I_d; \\
        \A_i \times [0, N_{0i}] &\text{if } i \in I_d
    \end{cases}
\end{equation*}
for the subsystem $ \tilde\Sigma_i $ of \eqref{eq:aug-inter} with rate coefficients
\begin{equation}\label{eq:aug-inter-adt-sub-lya-rate-coef}
    \begin{cases}
        \tilde c_i := c_i,\, \tilde d_i := d_i &\text{if } i \notin I_d; \\
        \tilde c_i := c_i - L_i \delta_i,\, \tilde d_i := d_i + L_i &\text{if } i \in I_d.
    \end{cases}
\end{equation}
More specifically,
\begin{enumerate}[1.]
    \item there exist functions $ \tilde\psi_{i1}, \tilde\psi_{i2} \in \Kinf $ such that
        \begin{equation}\label{eq:aug-inter-adt-sub-lya-bnd}
            \tilde\psi_{i1}(|z_i|_{\tilde\A_i}) \leq W_i(z_i) \leq \tilde\psi_{i2}(|z_i|_{\tilde\A_i}) \quad \forall z_i \in \cZ_i;
        \end{equation}
    \item there exist internal gains $ \tilde\chi_{ij} \in \K,\, j \neq i $ defined by
        \begin{equation}\label{eq:aug-inter-adt-sub-lya-gain-fcn-int}
            \tilde\chi_{ij}(r) :=
            \begin{cases}
                \chi_{ij}(r) &\text{if } i \notin I_d; \\
                e^{L_i N_{0i}} \chi_{ij}(r) &\text{if } i \in I_d
            \end{cases}
        \end{equation}
        with $ \chi_{ij} $ as in \eqref{eq:inter-sub-lya-gain} and $ \tilde\chi_{ii} \equiv 0 $, and an external gain $ \tilde\chi_i \in \K $ such that for all $ (z, u) \in \tilde\cC $ with $ z_i \notin \tilde\A_i $,
        \begin{equation}\label{eq:aug-inter-adt-sub-lya-gain}
            W_i(z_i) \geq \max \!\Big\{ \!\max_{j=1}^{n} \tilde\chi_{ij}(W_j(z_j)),\, \tilde\chi_i(|u|) \Big\}
        \end{equation}
        implies that
        \begin{equation}\label{eq:aug-inter-adt-sub-lya-flow}
            \dot W_i(z_i; y_i) \leq -\tilde c_i W_i(z_i) \qquad \forall y_i \in \tilde F_i(z, u);
        \end{equation}
    \item for all $ (z, u) \in \tilde\cD $,
        \begin{multline}\label{eq:aug-inter-adt-sub-lya-jump}
            W_i(y_i) \leq \max \!\Big\{ e^{-\tilde d_i} W_i(z_i),\, \max_{j=1}^{n} \tilde\chi_{ij}(W_j(z_j)), \\
            \tilde\chi_i(|u|) \Big\} \qquad \forall y_i \in \tilde G_i(z, u).
        \end{multline}
\end{enumerate}
\end{Aussage}
\begin{proof}
If $ i \notin I_d $, then the claim follows directly from the assumption that $ V_i $ is a candidate exponential ISS Lyapunov function with rate coefficients $ c_i, d_i $. Therefore, we only consider the case $ i \in I_d $ in the following proof. As $ V_i $ is locally Lipschitz outside $ \A_i $ and the map $ \tau_i \mapsto e^{L_i \tau_i} $ is smooth, $ W_i $ is also locally Lipschitz outside $ \tilde\A_i $.

First, consider the functions $ \tilde\psi_{i1}, \tilde\psi_{i2} \in \Kinf $ defined by
$$ \tilde\psi_{i1}(r) := \psi_{i1}(r), \quad \tilde\psi_{i2}(r) := e^{L_i N_{0i}} \psi_{i2}(r) $$
with $ \psi_{i1}, \psi_{i2} $ as in \eqref{eq:inter-sub-lya-bnd}. Then \eqref{eq:aug-inter-adt-sub-lya-bnd} follows from \eqref{eq:inter-sub-lya-bnd}.

Second, consider the function $ \tilde\chi_i \in \K $ defined by
\begin{equation}\label{eq:aug-inter-adt-sub-lya-gain-fcn-ext}
    \tilde\chi_i(r) := e^{L_i N_{0i}} \chi_i(r)
\end{equation}
with $ \chi_i $ as in \eqref{eq:inter-sub-lya-gain}. For each $ (z, u) \in \tilde\cC $ with $ z_i \notin \tilde\A_i $, if \eqref{eq:aug-inter-adt-sub-lya-gain} holds, then
\begin{align*}
    &V_i(x_i) = e^{-L_i \tau_i} W_i(z_i) \geq e^{-L_i N_{0i}} \max_{j=1}^{n} \tilde\chi_{ij}(W_j(z_j)) \\
    &= \max_{j=1}^{n} \chi_{ij}(W_j(z_j)) \geq \max_{j=1}^{n} \chi_{ij}(V_j(x_j)),
\end{align*}
and $ V_i(x_i) = e^{-L_i \tau_i} W_i(z_i) \geq e^{-L_i N_{0i}} \tilde\chi_i(|u|) = \chi_i(|u|) $. Hence \eqref{eq:inter-sub-lya-gain}, and therefore \eqref{eq:inter-sub-lya-flow}, holds. For all $ y_i \in \tilde F_i(z, u) $, let $ y_i = (y_{i1}, y_{i2}) $ be such that $ y_{i1} \in F_i(x, u) $ and $ y_{i2} \in [0, \delta_i] $. Following \eqref{eq:inter-sub-lya-flow}, \eqref{eq:inter-sub-lay-rate-coef}, and \eqref{eq:aug-inter-adt-sub-lya-rate-coef},
\begin{align*}
    &\dot W_i(z_i; y_i) = e^{L_i \tau_i} \dot V_i(x_i; y_{i1}) + L_i e^{L_i \tau_i} V_i(x_i)\, y_{i2} \\
    &\leq -c_i e^{L_i \tau_i} V_i (x_i) + L_i \delta_i e^{L_i \tau_i} V_i(x_i) = -\tilde c_i W_i(z_i).
\end{align*}
Finally, consider an arbitrary $ (z, u) \in \tilde\cD $. For all $ y_i \in \tilde G_i(z, u) $, let $ y_i = (y_{i1}, y_{i2}) $ be such that $ y_{i1} \in G_i(x, u) $ and $ y_{i2} = \tau_i - 1 $. From \eqref{eq:aug-inter-adt-sub-lya-rate-coef}, it follows that
\begin{equation*}
    e^{-\tilde d_i} W_i(z_i) = e^{-d_i - L_i + L_i \tau_i} V_i(x_i) = e^{L_i y_{i2} - d_i} V_i(x_i),
\end{equation*}
and from \eqref{eq:aug-inter-adt-sub-lya-gain-fcn-int} and \eqref{eq:aug-inter-adt-sub-lya-gain-fcn-ext}, it follows that $ \tilde\chi_{ij}(W_j(z_j)) = e^{L_i N_{0i}} \chi_{ij} (W_j(z_j)) \geq e^{L_i y_{i2}} \chi_{ij}(V_j(x_j)) $ for all $ j $ and $ \tilde\chi_i(|u|) = e^{L_i N_{0i}} \chi_i(|u|) \geq e^{L_i y_{i2}} \chi_i(|u|) $, respectively. Substituting the previous equations into \eqref{eq:inter-sub-lya-jump} gives \eqref{eq:aug-inter-adt-sub-lya-jump}.

Therefore, $ W_i $ is a candidate exponential ISS Lyapunov function w.r.t. $ \tilde\A_i $ for the augmented subsystem $ \tilde\Sigma_i $ of \eqref{eq:aug-inter} with the rate coefficients $ \tilde c_i, \tilde d_i $ defined by \eqref{eq:aug-inter-adt-sub-lya-rate-coef}.
\end{proof}
Proposition~\ref{prop:aug-inter-adt-sub-lya} shows that it is possible to make all $ \tilde d_i > 0 $ by choosing large enough scalars $ L_i,\, i \in I_d $, at the cost of decreasing the convergence rates of continuous dynamics (as $ \tilde c_i = c_i - L_i \delta_i $ in \eqref{eq:aug-inter-adt-sub-lya-rate-coef} above), and increasing the internal gains (as $ \tilde\chi_{ij}(r) = e^{L_i N_{0i}} \chi_{ij}(r) $ in \eqref{eq:aug-inter-adt-sub-lya-gain-fcn-int} above). Consequently, for large enough integers $ N_{0i} $, it is possible that the small-gain condition \eqref{eq:smg-dfn} holds for the gain operator $ \Gamma $ defined by \eqref{eq:gain-operator}, but not for $ \tilde\Gamma: \R_+^n \to \R_+^n $ defined by\footnote{However, if all the original internal gains $ \chi_{ij} $ are linear, and the gain matrix $ \Gamma_M $ defined by \eqref{eq:gain-mx} is a triangular matrix (i.e., if \eqref{eq:inter} is a cascade interconnection), then \eqref{eq:smg-dfn} always holds for $ \tilde\Gamma $, as all the cyclic gains equal zero.}
\begin{equation*}
    \tilde\Gamma(r_1, \ldots, r_n) := \Big( \max_{j=1}^{n} \tilde\chi_{1j}(r_j), \ldots, \max_{j=1}^{n} \tilde\chi_{nj}(r_j) \Big).
\end{equation*}
To see the consequence of this fact clearer, consider for simplicity an interconnection of two subsystems $ \Sigma_1, \Sigma_2 $, and their candidate exponential ISS Lyapunov functions $ V_1, V_2 $ with rate coefficients $ c_1, d_2 > 0 > d_1, c_2 $ and linear internal gains $ \chi_{12}, \chi_{21} > 0 $. After we augment $ \Sigma_1 $ with an ADT clock $ \delta_1 \in [0, N_{01}] $, the matrix $ \tilde\Gamma_M $ is given by
\begin{equation*}
    \tilde\Gamma_M =
    \begin{bmatrix}
        0 & \tilde\chi_{12} \\
        \tilde\chi_{21} & 0
    \end{bmatrix}
    =
    \begin{bmatrix}
        0 & e^{L_1 N_{01}} \chi_{12} \\
        \chi_{21} & 0
    \end{bmatrix},
\end{equation*}
and $ \rho(\tilde\Gamma_M) < 1 $ holds iff $ \chi_{12} \chi_{21} < e^{-L_1 N_{01}} $. In order to make the rate coefficient $ \tilde d_1 = d_1 + L_1 > 0 $, we need to choose a scalar $ L_1 > - d_1 $. Also, the integer $ N_{01} \geq 1 $. Hence we cannot apply Theorem~\ref{thm:inter-lya-lin} to the augmented interconnection unless the original internal gains $ \chi_{12}, \chi_{21} $ satisfy $\chi_{12} \chi_{21} \leq e^{d_1} < 1$.

The observation above hints that it may be better to make all $ \tilde c_i > 0 $ (instead of making all $ \tilde d_i > 0 $ as in this subsection). See \cite{YangLiberzonMironchenko2016} for a case-by-case study comparing the two schemes.

\subsection{Making continuous dynamics ISS}\label{ssec:aug-radt}
In the following, we construct candidate exponential ISS Lyapunov functions so that all rate coefficients $ \tilde c_i > 0 $.

We say that a solution pair $ (x, u) $ of \eqref{eq:inter} admits a reverse average dwell-time (RADT) \cite{HespanhaLiberzonTeel2008} $ \delta^* > 0 $ if there is an integer $ N_0^* \geq 1 $ so that all $ (s, k) \preceq (t, j) $ in $ \dom x $ satisfy
\begin{equation}\label{eq:radt}
    t - s \leq \delta^* (j - k) + N_0^* \delta^*.
\end{equation}
Following \cite[Appendix]{CaiTeelGoebel2008} and \cite[Section~IV.B]{LiberzonNesicTeel2014}, a hybrid time domain satisfies \eqref{eq:radt} iff it is the domain of an RADT clock $ \tau $ defined by
\begin{equation*}
\begin{aligned}
    &\dot\tau = 1, &\qquad &\tau \in [0, N_0^* \delta^*], \\
    &\tau^+ = \max\{0,\, \tau - \delta^*\}, &\qquad &\tau \in [0, N_0^* \delta^*].
\end{aligned}
\end{equation*}
Denote by $ I_c := \{i : c_i < 0\} $ the index set of subsystems with non-ISS continuous dynamics. Let $ z_i := x_i \in \cX_i =: \cZ_i $ for $ i \notin I_c $ and $ z_i := (x_i, \tau_i) \in \cX_i \times [0, N_{0i}^* \delta_i^*] =: \cZ_i $ with an integer $ N_{0i} \geq 1 $ for $ i \in I_c $. Consider the augmented interconnection $ \tilde\Sigma $ with state $ z := (z_1, \ldots, z_n) \in \cZ_1 \times \cdots \times \cZ_n =: \cZ $ and input $ u \in \cU $ modeled by \eqref{eq:aug-inter}, where $ \tilde\cC := \tilde\cC_1 \times \cdots \times \tilde\cC_n \times \cC_u $ with $ \tilde\cC_i = \cC_i $ for $ i \notin I_c $ and $ \tilde\cC_i = \cC_i \times [0, N_{0i}^* \delta_i^*] $ for $ i \in I_c $, $ \tilde\cD := \tilde\cD_1 \times \cdots \times \tilde\cD_n \times \cD_u $ with $ \tilde\cD_i = \cD_i $ for $ i \notin I_c $ and $ \tilde\cD_i = \cD_i \times [0, N_{0i}^* \delta_i^*] $ for $ i \in I_c $, $ \tilde F := (\tilde F_1, \ldots, \tilde F_n) $ with $ \tilde F_i(z, u) := F_i(x, u) $ for $ i \notin I_c $ and $ \tilde F_i(z, u) := F_i(x, u) \times \{1\} $ for $ i \in I_c $, and $ \tilde G := (\tilde G_1, \ldots, \tilde G_n) $ with $ \tilde G_i(z, u) := G_i(x, u) $ for $ i \notin I_c $ and $ \tilde G_i(z, u) := G_i(x, u) \times \{\max\{0,\, \tau_i - \delta_i^*\}\} $ for $ i \in I_c $. Then \eqref{eq:aug-inter} is a hybrid system with the data $ \tilde\cH := (\tilde F, \tilde G, \tilde\cC, \tilde\cD, \cZ, \cU) $. The dynamics of $ z_i $ is called the $ i $-th augmented subsystem of \eqref{eq:aug-inter} and is denoted by $ \tilde\Sigma_i $.

In the following proposition, we apply the modification technique from \cite[Proposition~IV.4]{LiberzonNesicTeel2014} to construct a candidate exponential ISS Lyapunov functions for each augmented subsystem $ \tilde\Sigma_i $ based on the candidate exponential ISS Lyapunov function for the subsystem $ \Sigma_i $ of the original interconnection \eqref{eq:inter} and the RADT clock $ \tau_i $.
\begin{Aussage}\label{prop:aug-inter-radt-sub-lya}
Consider a subsystem $ \Sigma_i $ of the original interconnection \eqref{eq:inter}. Suppose that it admits a candidate exponential ISS Lyapunov function $ V_i $ w.r.t. a set $ \A_i $ with rate coefficients $ c_i, d_i $. For a scalar $ L_i \geq 0 $, the function $ W_i: \cZ_i \to \R_+ $ defined by
\begin{equation}
    W_i(z_i) :=
    \begin{cases}
        V_i(x_i) &\text{if } i \notin I_c; \\
        e^{-L_i \tau_i} V_i(x_i) &\text{if } i \in I_c
    \end{cases}
\end{equation}
is a candidate exponential ISS Lyapunov function w.r.t. 
\begin{equation*}
    \tilde\A_i :=
    \begin{cases}
        \A_i &\text{if } i \notin I_c; \\
        \A_i \times [0, N_{0i}^* \delta_i^*] &\text{if } i \in I_c
    \end{cases}
\end{equation*}
for the augmented subsystem $ \tilde\Sigma_i $ of \eqref{eq:aug-inter} with rate coefficients
\begin{equation}\label{eq:aug-inter-radt-sub-lya-rate-coef}
    \begin{cases}
        \tilde c_i := c_i,\, \tilde d_i := d_i &\text{if } i \notin I_c; \\
        \tilde c_i := c_i + L_i,\, \tilde d_i := d_i - L_i \delta_i^* &\text{if } i \in I_c.
    \end{cases}
\end{equation}
More specifically,
\begin{enumerate}[1.]
    \item there exist functions $ \tilde\psi_{i1}, \tilde\psi_{i2} \in \Kinf $ so that \eqref{eq:aug-inter-adt-sub-lya-bnd} holds;
    \item there exist internal gains $ \tilde\chi_{ij} \in \K,\, j \neq i $ defined by\footnote{Note that in \eqref{eq:aug-inter-adt-sub-lya-gain-fcn-int}, the forms of the internal gains $ \tilde\chi_{ij} $ depend on whether $ i \in I_d $, while in \eqref{eq:aug-inter-radt-sub-lya-gain-fcn-int}, the forms of $ \tilde\chi_{ij} $ depend on whether $ j \in I_c $.\label{ftnt:aug-int-gain}}
        \begin{equation}\label{eq:aug-inter-radt-sub-lya-gain-fcn-int}
            \tilde\chi_{ij}(r) :=
            \begin{cases}
                \chi_{ij}(r) &\text{for } j \notin I_c; \\
                \chi_{ij}(e^{L_j N_{0j}^* \delta_j^*} r) &\text{for } j \in I_c
            \end{cases}
        \end{equation}
        with $ \chi_{ij} $ as in \eqref{eq:inter-sub-lya-gain} and $ \tilde\chi_{ii} \equiv 0 $, and an external gain $ \tilde\chi_i \in \K $ such that for all $ (z, u) \in \tilde\cC $ with $ z_i \notin \tilde\A_i $, \eqref{eq:aug-inter-adt-sub-lya-gain} implies \eqref{eq:aug-inter-adt-sub-lya-flow};
    \item for all $ (z, u) \in \tilde\cD $, \eqref{eq:aug-inter-adt-sub-lya-jump} holds.
\end{enumerate}
\end{Aussage}
\begin{proof}
If $ i \notin I_c $, then the claim follows directly from the assumption that $ V_i $ is a candidate exponential ISS Lyapunov function with rate coefficients $ c_i, d_i $. Therefore, we only consider the case $ i \in I_c $ in the following proof. As $ V_i $ is locally Lipschitz outside $ \A_i $ and the map $ \tau_i \mapsto e^{-L_i \tau_i} $ is smooth, $ W_i $ is locally Lipschitz outside $ \tilde\A_i $.

First, consider the functions $ \tilde\psi_{i1}, \tilde\psi_{i2} \in \Kinf $ defined by
$$ \tilde\psi_{i1}(r) := e^{-L_i N_{0i}^* \delta_i^*} \psi_{i1}(r), \quad \tilde\psi_{i2}(r) := \psi_{i2}(r) $$
with $ \psi_{i1}, \psi_{i2} $ as in \eqref{eq:inter-sub-lya-bnd}. Then \eqref{eq:aug-inter-adt-sub-lya-bnd} follows from \eqref{eq:inter-sub-lya-bnd}.

Second, consider the function $ \tilde\chi_i \in \K $ defined by
\begin{equation}\label{eq:aug-inter-radt-sub-lya-gain-fcn-ext}
    \tilde\chi_i(r) := \chi_i(r)
\end{equation}
with $ \chi_i $ as in \eqref{eq:inter-sub-lya-gain}. For each $ (z, u) \in \tilde\cC $ with $ z_i \notin \tilde\A_i $, if \eqref{eq:aug-inter-adt-sub-lya-gain} holds, then
\begin{align*}
    &V_i(x_i) = e^{L_i \tau_i} W_i(z_i) \geq W_i(z_i) \geq \max_{j=1}^{n} \tilde\chi_{ij}(W_j(z_j)) \\
    &= \max_{j=1}^{n} \chi_{ij}(e^{L_j N_{0j}^* \delta_j^*} W_j(z_j)) \geq \max_{j=1}^{n} \chi_{ij}(V_j(x_j)),
\end{align*}
and $ V_i(x_i) = e^{L_i \tau_i} W_i(z_i) \geq W_i(z_i) \geq \tilde\chi_i(|u|) \geq \chi_i(|u|) $. Hence \eqref{eq:inter-sub-lya-gain}, and therefore \eqref{eq:inter-sub-lya-flow}, holds. For all $ y_i \in \tilde F_i(z, u) $, let $ y_i = (y_{i1}, y_{i2}) $ be such that $ y_{i1} \in F_i(x, u) $ and $ y_{i2} = 1 $. Following \eqref{eq:inter-sub-lya-flow}, \eqref{eq:inter-sub-lay-rate-coef}, and \eqref{eq:aug-inter-radt-sub-lya-rate-coef},
\begin{multline*}
    \dot W_i(z_i; y_i) = e^{-L_i \tau_i} \dot V_i(x_i; y_{i1}) - L_i e^{-L_i \tau_i} V_i(x_i)\, y_{i2} \\
    \leq -c_i e^{-L_i \tau_i} V_i (x_i) - L_i e^{-L_i \tau_i} V_i(x_i) = -\tilde c_i W_i(z_i).
\end{multline*}
Finally, consider an arbitrary $ (z, u) \in \tilde\cD $. For all $ y_i \in \tilde G_i(z, u) $, let $ y_i = (y_{i1}, y_{i2}) $ be such that $ y_{i1} \in G_i(x, u) $ and $ y_{i2} = \max\{0,\, \tau_i - \delta_i^*\} $. From \eqref{eq:aug-inter-radt-sub-lya-rate-coef}, it follows that
$$ e^{-\tilde d_i} W_i(z_i) = e^{-d_i + L_i \delta_i^* - L_i \tau_i} V_i(x_i) \geq e^{-L_i y_{i2} - d_i} V_i(x_i), $$
and from \eqref{eq:aug-inter-radt-sub-lya-gain-fcn-int} and \eqref{eq:aug-inter-radt-sub-lya-gain-fcn-ext}, it follows that $ \tilde\chi_{ij}(W_j(z_j)) = \chi_{ij}(e^{L_j N_{0j}^* \delta_j^*} W_j(z_j)) \geq e^{-L_i y_{i2}} \chi_{ij}(V_j(x_j)) $ for all $ j $, and $ \tilde\chi_i(|u|) = \chi_i(|u|) \geq e^{-L_i y_{i2}} \chi_i(|u|) $, respectively. Substituting the previous equations into \eqref{eq:inter-sub-lya-jump} gives \eqref{eq:aug-inter-adt-sub-lya-jump}.

Therefore, $ W_i $ is a candidate exponential ISS Lyapunov function w.r.t. $ \tilde\A_i $ for the augmented subsystem $ \tilde\Sigma_i $ of \eqref{eq:aug-inter} with the rate coefficients $ \tilde c_i, \tilde d_i $ defined by \eqref{eq:aug-inter-radt-sub-lya-rate-coef}.
\end{proof}

\subsection{Example}\label{ssec:eg}
We demonstrate the approach of modifying ISS Lyapunov functions in a case where we cannot apply Theorem~\ref{thm:inter-lya} and Proposition~\ref{prop:hyd-iss-partial} to establish stability directly.

Consider an interconnection of two hybrid subsystems with the state $ x = (x_1, x_2) $ modeled by
\begin{equation*}
\begin{aligned}
    &\dot x_1 = x_1 + x_2^2,\quad \dot x_2 = -3 x_2 + 0.1 \sqrt{|x_1|}, &\qquad &x \in \cC, \\
    &x_1^+ = e^{-1} x_1,\quad x_2^+ = e x_2, &\qquad &x \in \cD,
\end{aligned}
\end{equation*}
where $ \cC = \cD = \R^2 $. It can be represented in the form of the general interconnection \eqref{eq:inter} without the external input $ u $ by letting $ n = 2 $, $ F_1(x) = x_1 + x_2^2 $, $ F_2(x) = -3 x_2 + 0.1 \sqrt{|x_1|} $, $ G_1(x) = e^{-1} x_1 $, and $ G_2(x) = e x_2 $. As $ \cC = \cD = \R^2 $, the system may flow or jump at any point in $ \R^2 $, and all solutions are complete. Hence the notions of pre-ISS and ISS coincide, and so do the notions of pre-GAS and GAS. The $ x_1 $-subsystem $ \Sigma_1 $ has stabilizing discrete dynamics but non-ISS continuous dynamics, while the $ x_2 $-subsystem $ \Sigma_2 $ has ISS continuous dynamics but destabilizing discrete dynamics. Thus we cannot apply Theorem~\ref{thm:inter-lya} and Proposition~\ref{prop:hyd-iss-partial} to establish pre-GAS of the interconnection directly.

Consider the functions $ V_1, V_2: \R \to \R_+ $ defined by
\begin{equation*}
    V_1(x_1) := |x_1|, \quad V_2(x_2) := |x_2|,
\end{equation*}
and the functions $ \chi_{12}, \chi_{21}: \R_+ \to \R_+ $ defined by
\begin{equation*}
    \chi_{12}(r) := r^2/a, \quad \chi_{21}(r) := \sqrt{r}/b
\end{equation*}
with some scalars $ a, b > 0 $. From
\begin{equation*}
\begin{aligned}
    V_1(x_1) \geq \chi_{12}(V_2(x_2)) &\implies \dot V_1(x_1) \leq (a + 1) V_1(x_1), \\
    V_2(x_2) \geq \chi_{21}(V_1(x_1)) &\implies \dot V_2(x_2) \leq (0.1b - 3) V_2(x_2),
\end{aligned}
\end{equation*}
and\footnote{Note that the discrete dynamics of both subsystems are autonomous, and hence we can ignore the terms corresponding to internal gains $ \chi_{12}, \chi_{21} $ in \eqref{eq:hyd-lya-alt-jump}. Similar simplifications will be made when we apply Proposition~\ref{prop:aug-inter-adt-sub-lya} and Theorem~\ref{thm:inter-lya}.\label{ftnt:eg}}
\begin{equation*}
    V_1(x_1^+) \leq e^{-1} V_1(x_1), \quad V_2(x_2^+) \leq e V_2(x_2)
\end{equation*}
for all $ x = (x_1, x_2) \in \R^2 $, it follows that $ V_1 $ and $ V_2 $ are candidate exponential ISS Lyapunov functions w.r.t. $ \{0\} $ for the subsystems $ \Sigma_1 $ and $ \Sigma_2 $ with the internal gains $ \chi_{12} $ and $ \chi_{21} $, respectively. Since the discrete dynamics of the $ \Sigma_2 $ is destabilizing, we invoke the modification scheme from Section~\ref{ssec:aug-adt}. Consider a solution $ x: \dom x \to \R^2 $ admitting an ADT $ \delta_2 > 0 $, that is, there exists an integer $ N_{02} \geq 1 $ such that all $ (s, k) \preceq (t, j) $ in $ \dom x $ satisfy
\begin{equation}\label{eq:eg-adt}
    j - k \leq \delta_2 (t - s) + N_{02}.
\end{equation}
The corresponding ADT clock $ \tau_2 $ is defined by
\begin{equation*}
\begin{aligned}
    &\dot\tau_2 \in [0, \delta_2], &\qquad &\tau_2 \in [0, N_{02}], \\
    &\tau_2^+ = \tau_2 - 1, &\qquad &\tau_2 \in [1, N_{02}].
\end{aligned}
\end{equation*}
Let $ z_1 := x_1 $ and $ z_2 := (x_2, \tau_2) $. Following Proposition~\ref{prop:aug-inter-adt-sub-lya}, the function $ W_2: \R \times [0, N_{02}] \to \R_+ $ defined by
\begin{equation*}
    W_2(z_2) := e^{L_2 \tau_2} V_2(x_2)
\end{equation*}
is a candidate exponential ISS Lyapunov function w.r.t. $ \tilde\A_2 := \{0\} \times [0, N_{02}] $ for the augmented subsystem $ \tilde\Sigma_2 $ with the internal gain $ \tilde\chi_{21} \in \K $ defined by
\begin{equation*}
    \tilde\chi_{21}(r) := e^{L_2 N_{02}} \chi_{21}(r) = e^{L_2 N_{02}} \sqrt{r}/b.
\end{equation*}
More specifically, for all $ (z_1, z_2) \in \R^2 \times [0, N_{02}] $, if
\begin{equation*}
    W_2(z_2) \geq \tilde\chi_{21}(V_1(z_1))
\end{equation*}
then
\begin{equation*}
\begin{aligned}
    \dot W_2(z_2; y_2) &= e^{L_2 \tau_2} \dot V_2(x_2) + L_2 e^{L_2 \tau_2} V_2(x_2) \dot\tau_2 \\
    &\leq (0.1b - 3) e^{L_2 \tau_2} V_2(x_2) + L_2 \delta_2 e^{L_2 \tau_2} V_2(x_2) \\
    &= (0.1b - 3 + L_2 \delta_2) W_2(z_2)
\end{aligned}
\end{equation*}
for all $ y_2 \in \{-3 x_2 + 0.1 \sqrt{|x_1|}\} \times [0, \delta_2] $. Furthermore,
\begin{equation*}
    W_2(e x_2, \tau_2 - 1) = e^{L_2 (\tau_2 - 1) + 1} V_2(x_2) \leq e^{1 - L_2} W_2(z_2)
\end{equation*}
(see also footnote~\ref{ftnt:eg}). To make the discrete dynamics of $ \tilde\Sigma_2 $ ISS, we set
\begin{equation}\label{eq:eg-const-bnd-lower}
    L_2 > 1.
\end{equation}
Following \eqref{eq:gain-operator}, the gain operator $ \tilde\Gamma: \R_+^2 \to \R_+^2 $ after modification is defined by
\begin{equation*}
    \tilde\Gamma(r_1, r_2) = (\chi_{12}(r_2), \tilde\chi_{21}(r_1));
\end{equation*}
thus the small-gain condition \eqref{eq:smg-dfn} holds for $ \tilde\Gamma $ iff
$\chi_{12}(\tilde\chi_{21}(r)) < r$ for all $r > 0$, 
or equivalently,
\begin{equation}\label{eq:eg-const-bnd-upper}
    L_2 < \dfrac{\ln(a b^2)}{2 N_{02}}.
\end{equation}
Let a scalar $ s > 0 $ be such that
$ e^{L_2 N_{02}}/b < 1/s < \sqrt{a}. $
Then $ \sigma := (\sigma_1, \sigma_2) $ with $ \sigma_1(r) := r, \ \ \sigma_2(r) := \tfrac{\sqrt{r}}{s} $
is an $ \Omega $-path w.r.t. the gain operator $ \tilde\Gamma $. Following Theorem~\ref{thm:inter-lya}, the function $ W: \R^2 \times [0, N_{02}] \to \R_+ $ defined by
{
\setlength{\abovedisplayskip}{5pt}
\setlength{\belowdisplayskip}{5pt}
\begin{equation*}
\begin{split}
    W(z) &:= \max\{\sigma_1^{-1}(V_1(z_1)),\, \sigma_2^{-1}(W_2(z_2))\} \\
    &\; = \max\{V_1(z_1),\, s^2 W_2(z_2)^2\}
\end{split}
\end{equation*}
}
is a candidate Lyapunov function w.r.t. $ \tilde\A := \{(0, 0)\} \times [0, N_{02}] $ for the augmented interconnection with state $ z := (z_1, z_2) \in \R^2 \times [0, N_{02}] =: \cZ $. More specifically, for all $ z \in \cZ $,
$$ \dot W(z; y) \leq -c W(z) $$
for all $ y \in \{-x_1 + x_2^2\} \times \{-3 x_2 + 0.1 \sqrt{|x_1|}\} \times [0, \delta_2] $ with
\begin{align*}
c := \min\{-(a + 1),\, 2(3 - 0.1b - L_2 \delta_2)\} < 0,
\end{align*}
where the inequality follows from $ a > 0 $. Furthermore,
{
\setlength{\abovedisplayskip}{5pt}
\setlength{\belowdisplayskip}{5pt}
$$ W(e^{-1} x_1, e x_2, \tau_2 - 1) \leq e^{-d} W(z) $$
}
with $d := \min\{1,\, 2 (L_2 - 1)\} > 0, $
which follows from \eqref{eq:eg-const-bnd-lower}.
Thus $ W $ is a candidate exponential Lyapunov function for the augmented interconnection with rate coefficients $ c, d $. Consider the set of solutions $ x: \dom x \to \R^2 $ admitting the ADT $ \delta_2 $ and also an RADT $ \delta^* > 0 $, that is, in addition to \eqref{eq:eg-adt}, there also exists an integer $ N_0^* \geq 1 $ such that all $ (s, l) \preceq (t, j) $ in $ \dom x $ satisfy
\begin{equation}\label{eq:eg-radt}
    t - s \leq \delta^* (j - k) + \delta^* N_0^*.
\end{equation}
Following Proposition~\ref{prop:hyd-iss-partial} and Remark~\ref{rmk:hyd-iss-partial-adt}, this set of solutions is GAS provided that
\begin{equation*}
    0 < \delta^* < \dfrac{d}{-c} = \dfrac{\min\{1,\, 2 (L_2 - 1)\}}{\max\{a + 1,\, 2(0.1b - 3 + L_2 \delta_2)\}}
\end{equation*}
and \eqref{eq:eg-const-bnd-upper} hold. For example, if $ a = 1 $, $ b = 5 $, and $ L_2 = 1.5 $, then the set of solutions satisfying the ADT condition \eqref{eq:eg-adt} with $ \delta_2 = 2.25 $ and $ N_{02} = 1 $, and also the RADT condition \eqref{eq:eg-radt} with $ \delta^* = 0.45 $ and $ N_0^* = 1 $ is GAS.

\section{Conclusion and future research}\label{sec:sum}
We have proved several small-gain theorems for interconnections of hybrid subsystems which yield candidate ISS Lyapunov functions for the interconnections. These results unify several Lyapunov-based small-gain theorems for hybrid systems \cite{NesicTeel2008,DashkovskiyKosmykov2013,LiberzonNesicTeel2014} and impulsive systems \cite{DashkovskiyKosmykovMironchenkoNaujok2012,DashkovskiyMironchenko2013SICON}, and pave the way to the following general scheme for establishing ISS of interconnections of hybrid subsystems:
\begin{enumerate}[1.]
    \item Construct a candidate exponential ISS Lyapunov function $ V_i $ for each subsystem $ \Sigma_i $ with rate coefficients $ c_i, d_i $ and linear internal gains.
    \item Compute the index sets $ I_d, I_c $ of non-ISS dynamics.
    \item Modify the candidate exponential ISS Lyapunov functions $ V_i $ \emph{either for all $ i \in I_d $} via Proposition~\ref{prop:aug-inter-adt-sub-lya} \emph{or for all $ i \in I_c $} via Proposition~\ref{prop:aug-inter-radt-sub-lya}.
    \item Invoke Theorem~\ref{thm:inter-lya-lin} to construct a candidate exponential ISS Lyapunov function $ W $ for the augmented interconnection $ \tilde\Sigma $ with rate coefficients $ c, d $.
    \item Derive the conditions for ISS of $ \tilde\Sigma $ via Proposition~\ref{prop:hyd-iss-partial}.
    \item Summarize the conditions for ISS of the original interconnection $ \Sigma $ from those in Steps~3 and~5.
\end{enumerate}
As we observed in Section~\ref{sec:aug}, the modification of candidate ISS Lyapunov functions in Step~3 leads to enlarged internal gains. Therefore, a considerable improvement of this scheme above lies in the fact that only the candidate ISS Lyapunov functions with indices from $ I_d $ or those with indices from $ I_c $ would be modified, instead of all those with indices from $ I_d \cup I_c $ as it was done in \cite{LiberzonNesicTeel2014}. If either $ I_d = \emptyset $ or $ I_c = \emptyset $, then no subsystem needs to be modified at all. Moreover, this scheme also applies to arbitrary interconnections composed of $ n \geq 2 $ subsystems.

In the scheme above, it is assumed that all $ V_i $ are candidate exponential ISS Lyapunov functions with linear internal gains. However, the modification also works for candidate exponential Lyapunov functions with \emph{nonlinear} internal gains, and Theorem~\ref{thm:inter-lya} was proved for arbitrary candidate ISS Lyapunov functions with nonlinear internal gains. If Proposition~\ref{prop:hyd-iss-partial} were extended to the case of non-exponential ISS Lyapunov functions, one could apply the scheme above for $ V_i $ with nonlinear internal gains as well. Such theorems have been proved in \cite[Theorems~1 and~3]{DashkovskiyMironchenko2013SICON} for impulsive systems, and we believe that they can be generalized to hybrid systems as well. This is one of the possible directions for future research.

The more challenging questions are whether one can establish ISS of an interconnection in the presence of destabilizing dynamics in subsystems without enlarging the internal gains, or without modifying ISS Lyapunov functions at all. At the time these questions remain open.

\begin{ack}
The work of A. Mironchenko was supported by the German Research Foundation (DFG) grant Wi 1458/13-1. The work of G. Yang and D. Liberzon was supported by the NSF grants CNS-1217811 and ECCS-1231196.

The authors thank Navid Noroozi for his comments on the proof of the main result. The authors are also grateful to the anonymous reviewers for their careful evaluation of the paper and valuable suggestions.
\end{ack}

\bibliographystyle{abbrv} 
\bibliography{reference} 

\begin{thebibliography}{10}

\bibitem{BermanPlemmons1994}
A.~Berman and R.~J. Plemmons.
\newblock {\em {Nonnegative Matrices in the Mathematical Sciences}}.
\newblock Society for Industrial and Applied Mathematics, 1994.

\bibitem{CaiTeel2009}
C.~Cai and A.~R. Teel.
\newblock {Characterizations of input-to-state stability for hybrid systems}.
\newblock {\em Systems {\&} Control Letters}, 58(1):47--53, 2009.

\bibitem{CaiTeelGoebel2007}
C.~Cai, A.~R. Teel, and R.~Goebel.
\newblock {Smooth Lyapunov functions for hybrid systems---Part I: Existence is
  equivalent to robustness}.
\newblock {\em IEEE Transactions on Automatic Control}, 52(7):1264--1277, 2007.

\bibitem{CaiTeelGoebel2008}
C.~Cai, A.~R. Teel, and R.~Goebel.
\newblock {Smooth Lyapunov functions for hybrid systems Part II:
  (Pre)Asymptotically stable compact sets}.
\newblock {\em IEEE Transactions on Automatic Control}, 53(3):734--748, 2008.

\bibitem{DashkovskiyEfimovSontag2011}
S.~Dashkovskiy, D.~V. Efimov, and E.~D. Sontag.
\newblock {Input to state stability and allied system properties}.
\newblock {\em Automation and Remote Control}, 72(8):1579--1614, 2011.

\bibitem{DashkovskiyKosmykov2013}
S.~Dashkovskiy and M.~Kosmykov.
\newblock {Input-to-state stability of interconnected hybrid systems}.
\newblock {\em Automatica}, 49(4):1068--1074, 2013.

\bibitem{DashkovskiyKosmykovMironchenkoNaujok2012}
S.~Dashkovskiy, M.~Kosmykov, A.~Mironchenko, and L.~Naujok.
\newblock {Stability of interconnected impulsive systems with and without time
  delays, using Lyapunov methods}.
\newblock {\em Nonlinear Analysis: Hybrid Systems}, 6(3):899--915, 2012.

\bibitem{DashkovskiyMironchenko2013}
S.~Dashkovskiy and A.~Mironchenko.
\newblock {Input-to-state stability of infinite-dimensional control systems}.
\newblock {\em Mathematics of Control, Signals, and Systems}, 25(1):1--35,
  2013.

\bibitem{DashkovskiyMironchenko2013SICON}
S.~Dashkovskiy and A.~Mironchenko.
\newblock {Input-to-state stability of nonlinear impulsive systems}.
\newblock {\em SIAM Journal on Control and Optimization}, 51(3):1962--1987,
  2013.

\bibitem{DashkovskiyRufferWirth2006}
S.~Dashkovskiy, B.~S. R{\"{u}}ffer, and F.~R. Wirth.
\newblock {On the construction of ISS Lyapunov functions for networks of ISS
  systems}.
\newblock In {\em 17th International Symposium on Mathematical Theory of
  Networks and Systems}, pages 77--82, 2006.

\bibitem{DashkovskiyRufferWirth2007}
S.~Dashkovskiy, B.~S. R{\"{u}}ffer, and F.~R. Wirth.
\newblock {An ISS small gain theorem for general networks}.
\newblock {\em Mathematics of Control, Signals, and Systems}, 19(2):93--122,
  2007.

\bibitem{DashkovskiyRufferWirth2010}
S.~Dashkovskiy, B.~S. R{\"{u}}ffer, and F.~R. Wirth.
\newblock {Small gain theorems for large scale systems and construction of ISS
  Lyapunov functions}.
\newblock {\em SIAM Journal on Control and Optimization}, 48(6):4089--4118,
  2010.

\bibitem{DesoerVidyasagar2009}
C.~A. Desoer and M.~Vidyasagar.
\newblock {\em {Feedback Systems: Input-Output Properties}}.
\newblock Society for Industrial and Applied Mathematics, 2009.

\bibitem{GoebelSanfeliceTeel2012}
R.~Goebel, R.~G. Sanfelice, and A.~R. Teel.
\newblock {\em {Hybrid Dynamical Systems: Modeling, Stability, and
  Robustness}}.
\newblock Princeton University Press, 2012.

\bibitem{Grune2002Book}
L.~Gr{\"{u}}ne.
\newblock {\em {Asymptotic Behavior of Dynamical and Control Systems under
  Perturbation and Discretization}}.
\newblock Springer Berlin Heidelberg, 2002.

\bibitem{HaddadChellaboinaNersesov2006}
W.~M. Haddad, V.~Chellaboina, and S.~G. Nersesov.
\newblock {\em {Impulsive and Hybrid Dynamical Systems}}.
\newblock Princeton University Press, 2006.

\bibitem{HespanhaLiberzonTeel2008}
J.~P. Hespanha, D.~Liberzon, and A.~R. Teel.
\newblock {Lyapunov conditions for input-to-state stability of impulsive
  systems}.
\newblock {\em Automatica}, 44(11):2735--2744, 2008.

\bibitem{HespanhaMorse1999}
J.~P. Hespanha and A.~S. Morse.
\newblock {Stability of switched systems with average dwell-time}.
\newblock In {\em 38th IEEE Conference on Decision and Control}, volume~3,
  pages 2655--2660, 1999.

\bibitem{Hill1991}
D.~J. Hill.
\newblock {A generalization of the small-gain theorem for nonlinear feedback
  systems}.
\newblock {\em Automatica}, 27(6):1043--1045, 1991.

\bibitem{JiangMareelsWang1996}
Z.-P. Jiang, I.~M.~Y. Mareels, and Y.~Wang.
\newblock {A Lyapunov formulation of the nonlinear small-gain theorem for
  interconnected ISS systems}.
\newblock {\em Automatica}, 32(8):1211--1215, 1996.

\bibitem{JiangTeelPraly1994}
Z.-P. Jiang, A.~R. Teel, and L.~Praly.
\newblock {Small-gain theorem for ISS systems and applications}.
\newblock {\em Mathematics of Control, Signals, and Systems}, 7(2):95--120,
  1994.

\bibitem{JiangWang2001}
Z.-P. Jiang and Y.~Wang.
\newblock {Input-to-state stability for discrete-time nonlinear systems}.
\newblock {\em Automatica}, 37(6):857--869, 2001.

\bibitem{KarafyllisJiang2007}
I.~Karafyllis and Z.-P. Jiang.
\newblock {A small-gain theorem for a wide class of feedback systems with
  control applications}.
\newblock {\em SIAM Journal on Control and Optimization}, 46(4):1483--1517,
  2007.

\bibitem{KarafyllisJiang2011}
I.~Karafyllis and Z.-P. Jiang.
\newblock {A vector small-gain theorem for general non-linear control systems}.
\newblock {\em IMA Journal of Mathematical Control and Information},
  28(3):309--344, 2011.

\bibitem{LailaNesic2003}
D.~S. Laila and D.~Ne{\v{s}}i{\'{c}}.
\newblock {Discrete-time Lyapunov-based small-gain theorem for parameterized
  interconnected ISS systems}.
\newblock {\em IEEE Transactions on Automatic Control}, 48(10):1783--1788,
  2003.

\bibitem{LiberzonNesic2006}
D.~Liberzon and D.~Ne{\v{s}}i{\'{c}}.
\newblock {Stability analysis of hybrid systems via small-gain theorems}.
\newblock In {\em Hybrid Systems: Computation and Control}, pages 421--435.
  Springer Berlin Heidelberg, 2006.

\bibitem{LiberzonNesicTeel2014}
D.~Liberzon, D.~Ne{\v{s}}i{\'{c}}, and A.~R. Teel.
\newblock {Lyapunov-based small-gain theorems for hybrid systems}.
\newblock {\em IEEE Transactions on Automatic Control}, 59(6):1395--1410, 2014.

\bibitem{LiuJiangHill2012}
T.~Liu, Z.-P. Jiang, and D.~J. Hill.
\newblock {Lyapunov formulation of the ISS cyclic-small-gain theorem for hybrid
  dynamical networks}.
\newblock {\em Nonlinear Analysis: Hybrid Systems}, 6(4):988--1001, 2012.

\bibitem{MareelsHill1992}
I.~M.~Y. Mareels and D.~J. Hill.
\newblock {Monotone stability of nonlinear feedback systems}.
\newblock {\em Journal of Mathematical Systems, Estimation, and Control},
  2:275--291, 1992.

\bibitem{Mironchenko2012}
A.~Mironchenko.
\newblock {\em {Input-to-state stability of infinite-dimensional control
  systems}}.
\newblock {Ph.D. Dissertation}, Universit{\"{a}}t Bremen, 2012.

\bibitem{MironchenkoYangLiberzon2014}
A.~Mironchenko, G.~Yang, and D.~Liberzon.
\newblock {Lyapunov small-gain theorems for not necessarily ISS hybrid
  systems}.
\newblock In {\em 21st International Symposium on Mathematical Theory of
  Networks and Systems}, pages 1001--1008, 2014.

\bibitem{MitraLiberzonLynch2008}
S.~Mitra, D.~Liberzon, and N.~Lynch.
\newblock {Verifying average dwell time of hybrid systems}.
\newblock {\em ACM Transactions on Embedded Computing Systems}, 8(1):1--37,
  2008.

\bibitem{Morse1996}
A.~S. Morse.
\newblock {Supervisory control of families of linear set-point
  controllers---Part I. Exact matching}.
\newblock {\em IEEE Transactions on Automatic Control}, 41(10):1413--1431,
  1996.

\bibitem{NesicLiberzon2005}
D.~Ne{\v{s}}i{\'{c}} and D.~Liberzon.
\newblock {A small-gain approach to stability analysis of hybrid systems}.
\newblock In {\em 44th IEEE Conference on Decision and Control}, pages
  5409--5414, 2005.

\bibitem{NesicTeel2008}
D.~Ne{\v{s}}i{\'{c}} and A.~R. Teel.
\newblock {A Lyapunov-based small-gain theorem for hybrid ISS systems}.
\newblock In {\em 47th IEEE Conference on Decision and Control}, pages
  3380--3385, 2008.

\bibitem{Ruffer2010}
B.~S. R{\"u}ffer.
\newblock {Monotone inequalities, dynamical systems, and paths in the positive
  orthant of Euclidean n-space}.
\newblock {\em Positivity}, 14(2):257--283, 2010.

\bibitem{Sontag1989}
E.~D. Sontag.
\newblock {Smooth stabilization implies coprime factorization}.
\newblock {\em IEEE Transactions on Automatic Control}, 34(4):435--443, 1989.

\bibitem{YangLiberzonMironchenko2016}
G.~Yang, D.~Liberzon, and A.~Mironchenko.
\newblock {Analysis of different Lyapunov function constructions for
  interconnected hybrid systems}.
\newblock In {\em 55th IEEE Conference on Decision and Control}, pages
  465--470, 2016.

\end{thebibliography}
\end{document}